\documentclass[12pt]{amsart} 

\usepackage{amsmath,euscript}
\usepackage{amsthm}
\usepackage{amssymb}
\usepackage{graphicx}
\usepackage{mathrsfs}
\usepackage{epsf}
\usepackage{psfrag}
\usepackage{enumerate}
\usepackage{amsfonts,latexsym,marginnote}
\usepackage{hyperref}

\usepackage[margin=1.12in]{geometry}
 
\newtheorem{thm}{Theorem}

\newtheorem{lemma}[thm]{Lemma}
\newtheorem{prop}[thm]{Proposition}
\newtheorem{cor}[thm]{Corollary}

\newtheorem{conj}{Conjecture}

\theoremstyle{remark}
\newtheorem{remark}[thm]{Remark}

\theoremstyle{definition}
\newtheorem{definition}{Definition}

\numberwithin{thm}{section}
\numberwithin{equation}{section}

\newcommand{\nc}{\newcommand}

\newcommand{\R}{\mathbb{R}}

\nc{\jx}{\langle x \rangle}

\nc{\comment}[1]{\vskip.3cm
\fbox{
\parbox{0.93\linewidth}{\footnotesize #1}}
\vskip.3cm}

\nc{\bS}{\mathbb S}

\nc{\al}{\alpha}
\nc{\bet}{\beta}
\nc{\del}{\delta}
\nc{\Del}{\Delta}
\nc{\G}{\Gamma}
\nc{\g}{\gamma}
\nc{\lam}{\lambda}
\nc{\Lam}{\Lambda}
\nc{\Om}{\Omega}
\nc{\om}{\omega}
\nc{\ta}{\tau}
\nc{\w}{\omega}
\nc{\io}{\iota}
\nc{\h}{\theta}
\nc{\z}{\zeta}
\nc{\s}{\sigma}
\nc{\Si}{\Sigma}
\nc{\e}{\epsilon}

\renewcommand{\k}{\kappa}

\newcommand{\ka}{\kappa}

\nc{\ran}{\rangle}
\nc{\lan}{\langle}
\newcommand{\re}{\operatorname{Re}}
\newcommand{\im}{{\rm Im}}

\newcommand{\Tr}{\mathrm{Tr}}
\newcommand{\tr}{\mathrm{Tr}}

\nc{\bfone}{{\bf 1}}
\nc{\1}{{\bf 1}}

\newcommand{\ra}{\rightarrow}

\newcommand{\ls}{\lesssim}

\newcommand{\p}{{\partial}}
\newcommand{\n}{{\nabla}}

\def\<{\langle}
\def\>{\rangle}

\usepackage{amssymb} 
\usepackage{graphicx} 

\newcommand{\den}{\operatorname{den}}

\newcommand{\DETAILS}[1]{}
\newcommand{\SID}[1]{}

\title[Time-dependent density functional theory]{Long-time behaviour of time-dependent 
\\ density functional theory}

\author{Fabio Pusateri}
\address{Fabio Pusateri,  Department of Mathematics, University of Toronto, 40 St. George street, Toronto, 
  M5S 2E4, Ontario, Canada}
\email{fabiop@math.toronto.edu}

\author{Israel Michael Sigal}
\address{Israel Michael Sigal, Department of Mathematics, University of Toronto, 40 St. George street, Toronto, 
  M5S 2E4, Ontario, Canada}
\email{im.sigal@utoronto.ca}

\begin{document}

\begin{abstract}
The density functional theory (DFT) is a remarkably successful theory of electronic structure of matter. 
At the foundation of this theory lies the Kohn-Sham (KS) equation.
In this paper, we describe the long-time behaviour of the time-dependent KS equation.
Assuming weak self-interactions, we prove 
global existence and scattering 
in (almost) the full ``short-range'' regime. 
This is achieved with new and simple techniques, naturally compatible with the structure of the DFT and 
involving commutator vector fields and non-abelian versions of Sobolev-Klainerman-type spaces
and inequalities.
%
\end{abstract}

\maketitle

\tableofcontents

\bigskip
\section{Introduction}

\medskip
\subsection{The DFT equation}
The density functional theory (DFT) is a remarkably successful theory of electronic structure of matter 
(see e.g. \cite{Jones, Kohn, KL, Lions, CLeBL} for some reviews). 
It naturally applies not just to electrons, but  to any fermion gas, say, of atoms, 
molecules or nucleons considered as point particles.   

At the foundation of the DFT is the seminal Kohn-Sham equation (KSE). Originally written in the stationary context 
and for pure states (represented through the Slater determinant by orthonormal systems of $n$ functions, 
called orbitals), 
the KSE has a natural extension to the time-dependent 
framework (see e.g. \cite{BCS, Burk, BWG, CS1, GR, Ull}). 
Moreover, it can be rewritten in terms of orthogonal projections and then extended to density operators, 
i.e. positive, trace-class operators (see for example \cite{ChPavl,CS1,GS,LewSab1} and references therein) 
and takes the following form:
\begin{align}\label{KS}
\frac{\p \g}{\p t} = i [h_{\g},\g], \qquad h_{\g}:=-\Delta + f(\g),
\end{align}
with $\g=\g(t)$ a {\it positive} operator-family  on $L^2(\R^d)$ and $f$ mapping a class of self-adjoint operators on $L^2(\R^d)$ into itself. 
In addition, we require 
that $f$ depends on $\g$ through the function 
\[\rho_\g(x, t):=\g(x, x, t),\] 
 where $\g(x, y, t)$ are integral kernels of $\g(t)$, i.e. 
 there is  $g: L^1_{\rm loc}(\R^d, \R) \ra L^2_{\rm loc}(\R^d, \R)$, 
 such that \[f(\g)=g(\rho_\g),\] 
where $g(\rho)$ on the right-hand side is considered as a multiplication operator.
$L^2(\R^d)$ is called the one-particle space, $\g(t)$, the {\it density operator} at time $t$ and $\rho_\g$, the {\it one-particle charge density}.

Initial conditions are taken to be 
non-negative operators, $\g|_{t=0}=\g_0\ge 0$ and, for fermions, 
in addition, satisfying $\g_0 \le 1$, which encodes the Pauli exclusion principle.
It is easy to show that, under suitable conditions 
on $g$, 
the solutions have the same properties,  
  $0\le \g\ (\le 1)$. (In fact, all eigenvalues of $\g$ are conserved under the evolution.)
\DETAILS{The density operator, $\g$, satisfies \begin{equation} \label{gam-prop} 
	0\le \gamma=\g^*\ ( \le 1)\end{equation}
where the second inequality is required only for fermions and it encodes the Pauli exclusion principle.} 

Since $\g \ge 0$, we have that  $\rho_\g(x, t):=\g(x, x, t)\ge 0$ 
and, since it is interpreted as the one-particle (charge) density,  $\tr \g =\int \rho_\g dx$ is the total number of particles.

Since $h_{\g}$ depends on $\g$ only through the density $\rho_\g$, 
$h_{\g}\equiv h(\rho_\g)$, Eq. \eqref{KS} is equivalent to the equation  for the density $\rho$,
\[\frac{\p \rho}{\p t} = \den(i [h(\rho),\g]), 
\]
where $\den(A)\equiv \rho_A$, the density for an operator $A$. 
Hence is the term {\it density functional theory} (DFT).

Note that \eqref{KS} with $f(\g)=g(\rho_\g)$ presents a natural extension of the 
generalized nonlinear Schr\"odinger equation with the nonlinearity given by $g(|\psi|^2)$, which includes 
both the local power nonlinearity and nonlocal Hartree one, see \eqref{g-ex1}. 
Indeed, if the initial condition, $\g|_{t=0}=\g_0$, is a rank-one projection, then so is the solution $\g(t)$.   
Applying the resulting equation 
to an arbitrary vector, we arrive at the generalized nonlinear Schr\"odinger equation for the unknown. 
(In this case the conserved $L^2-$norm equals one.) 

We assume that the nonlinearity or self-interaction $f(\g)$ 
is {\it translation, rotation and gauge covariant} in the sense that
\begin{align}\label{nonlin-covar}
U_{\lam} f(\g)U_{\lam}^{-1}= f(U_{\lam} \g U_{\lam}^{-1}),  
\end{align}
where $U_{\lam}$ is either the translation, rotation and gauge transformation, respectively given by
\begin{align*}
\begin{split}
& U_{\lam}^{\mathrm{tr}} : f(x) \mapsto f(x+\lambda), \qquad \lam \in \R^d,
\\
& U_{\lam}^{\mathrm{rot}} : f(x) \mapsto f(\lambda^{-1}x), \qquad \lam \in O(d),
\\
& U_{\lam}^{\mathrm{g}}  : f(x) \mapsto e^{i\lam}f(x), \qquad \lam \in \R.
\end{split}
\end{align*}

For $f(\g)=g(\rho_\g)$,  \eqref{nonlin-covar} and $U_{\lam}\rho_\g=\den(U_{\lam}\g U_{\lam}^{-1})$ 
imply that  $g(\rho_\g)$ satisfies 
\begin{align}\label{nonlin-covar'}
U_{\lam} g (\rho)U_{\lam}^{-1}= g (U_{\lam}\rho).
\end{align}
Here $g(\rho)$ is considered as a multiplication operator, and $\rho$ as a function. 

A standard example of self-interaction in physics 
 is the sum of a 
{\it Hartree-type} nonlinearity and a local {\it exchange-correlation term} of the form
\begin{align}\label{g-ex1}
g(\rho) = v\ast\rho + \mathrm{xc}(\rho)
\end{align} 
for some potential $v=v(x)$ and some function $\mathrm{xc}= \mathrm{xc}(\rho)$.
Important cases of $v$ in \eqref{g-ex1} are $v(x)=\lam/|x|$ (the Coulomb or Newton potential, if $d=3$) 
and $v(x)=\lam\del(x)$ (the local potential, which can also considered as part of the exchange term).
An important example of exchange-correlation term is the Dirac one, 
$\mathrm{xc}(\rho)=-c \rho^{1/3}$, $c> 0$, in $3$ dimensions.

In Subsection \ref{sec:Res} we will define a general class of self-interactions that we are going to consider.

\DETAILS{The Hartree equation \eqref{HE} can be extended naturally to the density functional model (DFT) 
by adding to $h_{\g}$ an exchange energy term, $ex(\rho_\gamma)$, which is a local 
   function of the function $\rho_\gamma(x):= \gamma(x, x)$, say, 
   the G\^ateaux derivative of the Dirac exchange energy $Ex(\rho)=-c \int\rho^{4/3}$. 
In other words, one considers more geral self-consistent hamiltonian
  \begin{align}\label{hgam-gen}h_{\g}:=-\Delta + g(\rho_\g),\ \text{ or, }\ h_{\rho}:=-\Delta + g(\rho),\end{align} 
  with 
  $g(\rho)$ real, Galilei and gauge invariant functional, 
  e.g.  $g(\rho) =v *\rho + ex(\rho)$,  or  $g(\rho)= \lam\rho+ ex(\rho)$, and, say $ex(\rho)=-c \rho^{1/3}$.} 
 


In general, one would like to address the following problems 

\setlength{\leftmargini}{2em}
\begin{itemize}
\item Global existence vs blowup; 
\item Asymptotic behaviour as $t\ra \infty/T_{\rm blowup}$ (scattering theory, return to equilibrium vs. blowup dynamics);
\item Static, self-similar and travelling wave solutions and their stability; 
\item Macroscopic limit (effective equations).
\end{itemize}

The first rigorous results were obtained in \cite{LiebSim,Lions1,Lions}, see \cite{LeBL,Lieb}.
The existence theory for the standard Hartree and Hartree-Fock equations 
(which are similar and closely related to \eqref{KS}) and for the Kohn-Sham equation
with trace class initial data, $\tr \g_0<\infty$, was developed in \cite{BoveDPF,Chad,ChadGl,Jerome}.
(See \cite{ACV} and references therein for results on the related Schr\"odinger-Poisson system.)

For the Hartree equation Lewin and Sabin \cite{LewSab1,LewSab2}
studied the harder case of non-trace class solutions.
For initial conditions given by suitable trace-class perturbations of translation invariant states $\gamma_f = f(-\Delta)$,
the authors established
global well-posedness \cite{LewSab1} in dimensions $d=2,3$ as  well as dispersive properties of the solutions and scattering for $d=2$ \cite{LewSab2}.
%
%
These results have been extended to 
the more singular case of local nonlinearities ($v(x)=\lam\del(x)$) 
by Chen, Hong and Pavlovi\'c, who proved global well-posedness 
in dimensions $d=2,3$ and zero temperature \cite{ChPavl}.
The same authors also proved scattering results 
in the case of dimension $3$ and higher \cite{ChPavl2}, left open in \cite{LewSab2}. 
Finally, we mention the recent work of Collot and de Suzzoni \cite{CollS} 
who proved analogues of the results of \cite{LewSab2,ChPavl2}
for the Hartree equation for random fields in $d\geq 4$.

For classical papers on scattering theory for Schr\"odinger and Hartree type evolution equations we refer to
Strauss \cite{Strauss} and Ginibre-Velo \cite{GV1}. 
See also the works \cite{HN,KP} and references therein, on the scattering critical cases 
and the work on the Chern-Simons-Schr\"odinger equation by Oh and the first author \cite{CSS},
where weighted energy estimates are done covariantly, by adapting the 
standard Schr\"odinger ``vector field'' (see $j_\ell$ in \eqref{J'})
to the covariant structure of the equations.
See also \cite{Sig,SigSof} on the use of related ideas 
in the context of quantum scattering theory.

\medskip
\subsection{Results}\label{sec:Res} 

 
For $p\in[1,\infty]$, we let $L^p(\R^d)$ be the standard Lebesgue spaces on $\R^d$ with the norms  denoted by
$\| \cdot \|_{L^p}$ or $\| \cdot \|_p$.
We also let $ L^r_w$ denote the weak $L^r$ space.
We assume that $f(\gamma)$ is of the form 
\begin{align}\label{SI1}
f(\g)= g(\rho) = \lambda_1 v \ast \rho + \lambda_2 \rho^\beta, 
\end{align}
with
\begin{align}\label{SI2}
v \in L^r_w(\R^d), \quad r\in(1,d), \qquad \mbox{and} \qquad
\beta > 1/\min(d, 2).
\DETAILS{\left\{ \begin{array}{lr}
\beta > 1/2, & d=2,3
\\
\beta > 1, & d=1.
\end{array}\right.}
\end{align}
Note that the convolution term is omitted for $d=1$.

To keep the exposition simple we let $d\leq3$, and will make a few comments about extensions 
to the higher dimensional cases in Remarks \ref{rem:d>3} and \ref{rem:d>3'} below.

Let $I^r$ denote the space of bounded operators satisfying 
\begin{align}\label{defIr}
{\|A\|}_{I^r}:=(\tr(A^*A)^{r/2})^{1/r}<\infty,
\end{align}
a trace ideal or non-commutative $L^r-$space. 

We say that Eq \eqref{KS} is {\it asymptotically complete}, or has the {\it short-range scattering property}, 
if and only if, for any initial condition $\g_0\in I^1$, there is an operator $\g_\infty\in I^1$ independent of $t$, 
such that the solution $\g(t)$ to equation \eqref{KS} satisfies, as $t\ra \infty$,  
\begin{align}\label{SRscat}
{\big\| \g(t) - e^{i t \Delta} \g_\infty e^{-i \Delta t} \big\|}_{I^1} \ra 0.
\end{align}   
Our main result is


\begin{thm} \label{thm:GWP-Scatt}
Consider \eqref{KS} with $g(\rho)$, satisfying \eqref{SI1}-\eqref{SI2}, and
initial datum $\gamma_0 = \gamma_0^\ast = \gamma(t=0)$ satisfying 
\begin{align}\label{g_0assumption}
\|\jx^b\g_0 \jx^b\|_{I^1} + \| \langle \nabla \rangle^b \g_0 \langle \nabla \rangle^b\|_{I^1} < \infty,
\end{align}
for some integer $b > d/2$.
Then, for $|\lambda_1|,|\lambda_2|$ sufficiently small depending on $\|\jx^b\g_0 \jx^b\|_{I^1}$,
the equation \eqref{KS} with the initial datum $\gamma_0$
\begin{itemize}
\item[(i)] is globally well-posed;
\item[(ii)] is asymptotically complete (see \eqref{SRscat}).
\end{itemize}
\end{thm}

Theorem \ref{thm:GWP-Scatt} follows from Theorem \ref{thm:GWP-ST-ka}  and Proposition \ref{prop:ka-to-gam} formulated in Section \ref{SecStra} below.
Statement (i) in Theorem \ref{thm:GWP-Scatt} is not new, and can be obtained under milder assumptions
on the data. 
The main new result is the scattering property, (ii). 
Importantly, our results are given in the natural (weighted) trace norms.

The class of self-interactions that we actually treat is larger than \eqref{SI1}--\eqref{SI2},
see Remark \ref{remSI} below.

For convolution potentials that have stronger integrability properties and no exchange terms, 
scattering results
also follow from the cited works \cite{LewSab2,ChPavl,ChPavl2}.
Here we are able to cover the full subcritical range
for the convolution part, and the almost full subcritical range for the xc term. 
As a byproduct of our proof, we obtain that the solutions given in Theorem \ref{thm:GWP-Scatt} 
also enjoy the local decay estimate \eqref{loc-dec'}.

In view of our analysis it is natural to formulate the following conjectures.

  
\DETAILS{ In the long-range case, we replace \eqref{SRscat} by something like 
 \begin{align}\label{LRscat}& 
 \| \g(t) - U_{\rm gauge} U_{t} f U_{t}^{-1} U_{\rm gauge}^{-1} \|_{I^1}\ra 0,
 \end{align}   
as $t\ra \infty$, where $U_{t}$ is defined after \eqref{nonlin-scal} ($U_{\lam}: \psi(x) \ra \lam^{d}\psi(\lam x)$) and $f$ is a time-independent operator satisfying a certain nonlinear equation.}

\begin{conj}[Exponent $\beta$]\label{Conj1}
The range of exponents $\beta$ in \eqref{SI2}
for which short-range scattering holds is \[\beta > 1/d.\]
\end{conj}

In this respect our result is sharp in dimensions $1$ and $2$, see \eqref{SI1}, 
while it is not optimal for $d=3$ (and $d>3$, see Remark \ref{rem:d>3}).

\begin{conj}[Scattering critical case]\label{Conj2}
For \eqref{SI1} with $v=|x|^{-\alpha}$,
\[g(\rho) = \lambda_1 {|x|}^{-\alpha} \ast \rho + \lambda_2 \rho^\beta 
  \quad \mbox{with} \quad 0<\alpha \leq 1, \quad 0<\beta \leq 1/d,\]
modified scattering holds.
In particular, for $\alpha=1$ and $\beta=1/d$, we expect that
\begin{align*}
{\big\| \g(t) -
  e^{-i(-t\Delta +g_\infty(-i\nabla) \, \log t)} \g_\infty e^{i(-t\Delta +g_\infty(-i\nabla) \, \log t)}
  \big\|}_{I^1} \ra 0,
\end{align*} 
as $t\ra \infty$ (with some algebraic rate), where 
$\g_\infty \in I^1$ is time independent, and $g_\infty$ is a time-independent operator
which depends nonlinearly on $\g_\infty$.


\end{conj}

We 
also mention two further natural and important directions for extending our results:
density operators with infinite traces (i.e. states with infinite average numbers of particles), 
and systems without translation or/and gauge symmetry. 
Concerning the first direction, one could approximate - in a controllable way -
infinite-trace density operators by finite-trace ones,  as it is done for the nonlinear Schr\"odinger equation 
with either the $L^2$-norm or the energy infinite. 
For the second direction, we believe our techniques could be adapted to deal with external potentials, 
at least with repulsive ones, at first. 
Moreover, translation invariance can often be restored by `giving external potentials dynamics', 
i.e. by considering translation invariant 
multicomponent systems consisting of different kinds of particles interacting with each other. 







Unlike most of the previous research on the Hartree, Hartree-Fock and Kohn-Sham equations, 
which uses the formulation of the equations in terms of the eigenfunctions of $\g$, 
we deal with the operator $\g$ directly. 
There are three basic ingredients in our approach: 

\setlength{\leftmargini}{2em}
\begin{itemize}

\item[(i)] Passing to the Hilbert space
$I^2$  of Hilbert-Schmidt operators with the inner product
\begin{align}\label{HSinner}
\lan \k,  \k' \ran_{I^2}:=\tr (\k^* \k').
\end{align}
by going from $\g$ to, roughly speaking, $\k:=\sqrt \g$;


\item[(ii)] Deriving almost conservation laws for non-abelian analogues of  weighted Sobolev spaces 
\begin{align}
\label{LOC2}
W^s &:= \big\{ \k \in I^2\, : \,  \sum_{|\alpha| \leq s} \| J^\alpha \k \|_{I^2} < \infty \big\},
\end{align}
 based on the space $I^2$ 
with the smoothness grading provided by operators 
\begin{align}\label{J'}
J_\ell \k & := [j_\ell ,\k], \quad \qquad j_\ell:=x_\ell-2 p_\ell t, \qquad p_\ell:=- i \p_\ell.
\end{align}
Here, as usual, $J^\alpha:=\prod_i J^{\al_i}_{\ell_i}$ for $\alpha:=(\al_i)$. Note that $W^0 =  I^2$.
Since $J_\ell$ is self-adjoint on (a dense subset of) $I^2$, one can define $J_\ell^r$,  
for general non-integer $r$, by the operator calculus.
In this paper, however, we will only use the spaces spaces $W^s$ with integer $s$.

\item[(iii)] Using a new class of local norms for Hilbert-Schmidt operators,
and establishing a non-commutative version of Gagliardo-Nirenberg-Klainerman-type estimates
which yield bounds on these local norms of $\k$, and eventually imply the desired estimates on $\g$.
\end{itemize}


\DETAILS{The DFE has the following properties:

\begin{itemize}

\item  
Galilean invariance; 

\item Gauge  invariance of the nonlinearity;


\item Conservation of energy and number of particles; 

\item Preservation of positivity;

 
\item Hamiltonian structure;

\item Scattering dimension;

\item Translationally invariant static solutions.  
 
\end{itemize}
}

\medskip
\begin{remark}[Class of nonlinearities]\label{remSI}
We can treat a wider class of self-interactions than \eqref{SI1}-\eqref{SI2}:
\begin{align}\label{g1g2}
f(\gamma) = g(\rho_\g) 
	= \lambda_1 g_1(\rho) + \lambda_2 g_2(\rho)
\end{align}
with $\lambda_1,\lambda_2 \in \R$ and the following assumptions on $g_1$ and $g_2$: 


\setlength{\leftmargini}{2em}
\begin{align}\label{g-cond0}
& \| g_1 (\rho)\|_\infty  \lesssim \| \rho\|_q^a, 
\quad \mbox{for some} \quad a\geq 1 \quad \mbox{with} \quad  a(1 -1/q) > 1/d,
\end{align}
and its G\^ateaux derivatives\footnote{The 
$k$-th G\^ateaux derivative could be defined by induction either as
\[d^k  g_1 (\rho) \xi_1 \xi_2 \dots \xi_k:= d(d^{k-1} g_1 (\rho) \xi_1 \xi_2 \dots \xi_{k-1})\xi_k,\] 
or 
\[d^k  g_1 (\rho) \xi_1 \xi_2 \dots \xi_k:= \prod_1^k\p_{s_j}|_{s_i=0\forall i}g(\rho+\sum_1^k s_j \xi_j).\]} 
satisfy
\begin{align}
\label{g-cond1}
& \|dg_1 (\rho) \xi\|_p  \lesssim \| \xi\|_q,   &  
\\ 
\label{g-cond2}
& \|d^2g_1 (\rho) \xi\eta\|_p  \lesssim  \| \xi\|_{q'}\| \eta\|_{q'}, &   
\end{align}
for some indexes $(p,q,q')$ with\footnote{There is no restriction on $q$ if $d=1,2$.} 
\begin{align}\label{pqcond1}
& p \geq d, \qquad q \leq \frac{d}{d-2}, \qquad 1\leq q' \leq 2,
\\
\label{pqcond}
&  1+1/p -1/q>1/d, \qquad 2+1/p -2/{q'}>1/d;
\end{align}
\medskip
and, for bounded $\rho$, 
\begin{align}
\label{g-cond0'}
& |g_2 (\rho)| \lesssim  \rho^\beta,\\ 
\label{g-cond1'}
& |d^k g_2 (\rho)| \lesssim  \rho^{\beta-k}, \qquad k=1,2,
\end{align}
where the power $\beta$ satisfies (notice the difference with \eqref{SI2})
\begin{align}
\label{g-condbeta}
\left\{ \begin{array}{lr}
\beta \geq 1/2, \quad & d=3
\\
\beta > 1/2, & d=2
\\
\beta > 1, & d=1.
\end{array}\right.
\end{align}

\end{remark}

\begin{remark}[Higher dimensions]\label{rem:d>3}
For $d\ge4$, we have to add conditions on the higher (G\^ateaux) derivatives of $g$.
For example, the natural generalization of \eqref{g-cond1}-\eqref{g-cond2} 
with \eqref{pqcond} would be the following assumption: there exist $(p,q,q_1,\dots,q_k)$ such that
\begin{align*}
& \|d^k  g_1 (\rho) \xi_1 \xi_2 \dots \xi_k \|_p  
	\lesssim \prod_{i=1}^k \| \xi_i \|_{q_i}, \qquad k+ 1/p -1/{q_1} \dots - 1/q_k>1/d,
\end{align*}
for all $k\le [d/2]+1$.

For $g_2$ one would instead assume \eqref{g-cond0'} and \eqref{g-cond1'} for all $k\le [d/2]+1$,
As for the restriction analogous to \eqref{g-condbeta}, 
our current argument would require $\beta \geq (1/2)[d/2]$; see also Remark \ref{rem:d>3'}.

\end{remark}

%
%
%

\medskip  
\begin{remark}[Non-self-adjoint extension]\label{rem:non-sa-gam} 
We considered \eqref{KS} on self-adjoint operators.  
 By extending $f(\g)$ to non-self-adjoint operators, we can extend \eqref{KS} to  non-self-adjoint $\g$'s.  
 Then, assuming $f(\g^*)=f(\g)^*$ (or $g(\bar\rho)=\bar g(\rho)$) 
 and extending condition \eqref{g-cond0} on $g$ appropriately, 
 we can show that \begin{align}\label{adj-prop}\al_t(\g^*)=\al_t(\g)^*,\end{align} 
 where $\al_t(\g_0):=\g(t)$, the solution to  \eqref{KS} with the initial condition. $\g(t=0)=\g_0$, 
 and, in particular, $\g_0^*=\g_0 \Longrightarrow \g(t)^*=\g(t)$.
\end{remark}

\medskip  
\begin{remark}
In the context of the Schr\"odinger evolution,
the operators $j_t:=x-2 p t,\  p:=- i \n$, have been used in several works to obtain a priori estimates on solutions,
see for example \cite{HN,KP,Sig,SigSof}.
\end{remark}  
\medskip  
\begin{remark}
The operators $j_t:=x-2 p t$ are the generators of the Galilean boost 
$U_{v, t}: \psi(x, t) \ra e^{i( v\cdot x -|v|^2 t)}\psi(x-2 vt, t)$, which can be written as 
\begin{align}\label{Uv} U_{v, t}: = e^{i(v\cdot x - |v|^2 t)}e^{- 2v\cdot \n t} =e^{i v\cdot (x -2p t)}=e^{i v\cdot j_t}.
\end{align}
(The second equality above follows from the Baker-Campbell-Hausdorff formula.)
This is lifted to a space of operators as
\begin{align}\label{gal-boost}
\k \ra U_{v, t}\k U_{v, t}^{-1} = e^{i  v\cdot J_t}\k.
\end{align}
\end{remark}


The paper is organized as follows.
In Section \ref{SecProperties} we recall several properties of \eqref{KS}
and give a definition of scattering criticality. 
In Section \ref{SecStra} we present our general strategy:
we introduce a ``half-density'' $\k$, 
such that $\kappa^\ast\k = \g$ and derive an equation for it;
we then state our main results concerning $\k$ and show how these imply Theorem \ref{thm:GWP-Scatt}.
Section \ref{sec:Pf-tildekap-est} contains the proof of a 
non-abelian version of a Gagliardo-Nirenberg-Klainerman-type inequality,
and Section \ref{sec:den-est} some simple estimates on densities.
In Section \ref{sec:approx-conservq} we estimate the evolution of the weighted energy,
and then use this in Section \ref{sec:ener-ineq-pf} to prove the main a priori bounds for the 
weighted norms of $\k$, see \eqref{LOC2}.
Finally, in Section \ref{sec:loc-exist} we prove a local existence result and continuity criterion for $\k$,
and combine it with the a priori weighted bounds to complete the proof of global well-posedness and scattering.

\medskip
\noindent
{\bf Acknowledgement}. 
The authors are grateful to Ilias Chenn, St\'ephane Nonnenmacher, 
Benjamin Schlein and Heinz Siedentop for useful discussions.
We are grateful to the referee for helpful comments, and for emphasizing the importance of the
problems discussed after Conjecture \ref{Conj2}.
The work on this paper was supported in part by NSERC Grant No. NA7901 (IMS), 
by a start-up grant from the University of Toronto and NSERC Grant No. 06487 (FP).

\bigskip
\section{Properties of KSE}\label{SecProperties}


\medskip
\paragraph{\bf Symmetries and conserved quantities.} 

The equation \eqref{KS} is invariant under the  {\it translation} and {\it rotation} transformations,  \begin{equation}\label{tr-rot-transf}
    T^{\rm trans}_h :  \g \mapsto U_h\g U_h^{-1}\ \text{ and }\ T^{\rm rot}_\rho : \g \mapsto U_\rho\g U_\rho^{-1} ,\end{equation}
 for any $h \in \R^d$ and
  any $\rho \in O(d)$. 
Here $U_h$ and $U_\rho$ are the standard translation and rotation transforms $U_h :    \phi(x) \mapsto \phi(x + h)$ and $U_\rho : \phi(x)  \mapsto  \phi(\rho^{-1}x)$.

 Note that \eqref{KS} has no gauge symmetry, unless it is coupled to the Maxwell equations.

The conserved energy and the number of particles are given by 
\begin{align}\label{H-en}
& E(\g) := \tr (h\g)+ G(\rho_\g),\\  
\label{numb}
&       N(\g) :=\tr \g=\int \rho_\g,
\end{align}
where $h:= -\Delta $ and, recall, $\rho_\g(x) :=\g(x, x)$ and $G(\rho)$ is an anti-$L^2-$gradient of $g(\rho)$, i.e. $d_\g G(\rho_\g)\xi =\Tr(g(\rho_\g)\xi)$. 
 $G(\rho)$  is the energy due to direct electrostatic self-interaction of the charge distribution $\rho_\g$ and the exchange-correlation energy, see \eqref{g-ex1}. 
If $g(\rho)=v* \rho+\lam \rho^\beta$, then 
\begin{align*}
G(\rho_\g) &= \frac12 \tr ((v* \rho_\g)\g)+ \frac{\lam}{\beta+1}\int \rho^{\beta+1}
  = \frac12 \int\rho_\g v* \rho_\g dx+ \frac{\lam}{\beta+1}\int \rho^{\beta+1}.
\end{align*} 


\medskip
\paragraph{\bf Hamiltonian structure.} 
  \eqref{KS} is a Hamiltonian system with the Hamiltonian given by the energy \eqref{H-en} and the Poisson bracket generated by the operator $J=J(\g): A\ra  i [A, \g]$ so that \eqref{KS} can be rewritten as 
\begin{align}\label{gHF-ham}
   \frac{\p \g}{\p t} =  J(\g)\n E(\g)
\end{align}
where $\n E(\g)$ is 
defined by $d E(\g)\xi= \Tr(\xi\n E(\g))$ with $d E(\g)$ being the G\^ateaux derivative of $E$.



\medskip
\paragraph{\bf Scattering criticality.} 
Consider the self-consistent hamiltonian 
\begin{align}\label{hgam-gen} 
h(\rho):=-\Delta + g(\rho),\end{align}
Let $f_t(x):=t^{-d} f(\frac xt)$ . We say that $g(\rho)$ is {\it short range} ({\it scattering subcritical}) if and only if 
\begin{align}\label{scattsub}
\int_1^\infty \| g(f_t)\|_\infty dt <\infty 
\end{align}
and {\it long-range} ({\it scattering critical or supercritical}) otherwise.

 
\medskip
\paragraph{\bf Scaling property.}
Another way to define scattering 
 criticality is to use the {\it scaling property} of the nonlinearity. 
 Assuming $g (\rho)$ satisfies
 \begin{align}\label{nonlin-scal}
U_{\lam} g (\rho)U_{\lam}^{-1}=  \lam^{-\al} g (U_{\lam}\rho),
\end{align}
where $U_{\lam}: \psi(x) \ra \lam^{d}\psi(\lam x)$ and $g(\rho)$ is considered as a multiplication operator and $\rho$ as a function, we say that  $g (\rho)$ is  scattering subcritical, resp. critical or supercritical, 
if and only if $\al>1$, resp. $\al=1$ or $\al<1$. 

By the way of an example, the scaling property \eqref{nonlin-scal}
holds for  $g_1(\rho) = v \ast \rho $, with the convolution potential $v(x)=\lam/|x|^\al$, for $\al < d$, and $v(x)=\lam\del(x)$, for $\al= d$,
and for $g_2(\rho) = \rho^\beta$ with $\alpha  = \beta d$.

Thus $g_1(\rho) = \lam |x|^{-\al} \ast \rho$ and $g_2(\rho) = \rho^\beta$ are subritical (resp. critical, resp. supercritical) iff $\al>1$ and $\beta>1/d$  (resp. $\al=1$ and $\beta=1/d$, resp. $\al<1$ and $\beta<1/d$).

As communicated by Schlein \cite{SchleinRk}
the criticality of $|x|^{-\al} \ast \rho$ and $\rho^\beta$ are related
since $c\rho^{\al/d}$ is the semi-classical limit of $|x|^{-\al} \ast \rho$.

Note that if $g$ is scattering sub-critical/critical/supercritical in the scaling sense
then it is in the same class in the sense of \eqref{scattsub}. Indeed, if $g$ satisfies  \eqref{nonlin-scal}, then
\begin{align*}
 g(f_t)  = g (U_{\frac1t}\rho) =t^{-\al} U_{\frac1t} g (\rho)U_{\frac1t}^{-1},
\end{align*}
and therefore $\| g(f_t)\|_\infty=t^{-\al}\| g(f)\|_\infty $ is integrable at $t=\infty$ iff $\al >1$.


\begin{remark}
For $g (\rho)$ satisfying \eqref{nonlin-scal} and $g (\lam\rho)=\lam^\nu g (\rho)$, 
\eqref{nonlin-scal} has scaling  covariance with respect to the operator 
$U_{\lam}^\beta: = \lam^{-\beta}U_{\lam}$, for an appropriate $\beta$. 
\end{remark}

\bigskip
\section{Strategy and main propositions}\label{SecStra}

\subsection{Passing to a Hilbert space} 
\label{sec:HilbSp} 
To work on a Hilbert space, we pass from $\g$ to $\sqrt \g$, or more precisely to $\k$, such that $\k^\ast\k=\g$. 
One can think of $\ka$ as a sort of ``half-density''.  
Then the KS \eqref{KS} becomes
\begin{align}\label{KS-k}
   \p_t \k  = i [h(\rho_{\k^*\k}),\k]
\end{align}
where, recall, $h(\rho):=-\Delta + g(\rho)$, 
and $\rho_{\g}(x, t):=\g(x, x, t)$. 
Equation \eqref{KS-k} will be the main focus for our proofs.
Note that if $\g = \k^\ast\k$ is trace-class, then $\k$ is a Hilbert-Schmidt operator. 

For brevity, we will use the notation  $h_{\g}\equiv h(\rho_\g)$ and, for complicated $\g$, the notation, 
$ \rho(\g)\equiv \rho_\g$. We have


\begin{prop} \label{prop:passing-HS}
Assume \eqref{g1g2}-\eqref{g-cond1'}. Then \eqref{KS-k} 
$\iff$  \eqref{KS}, 
in the sense that if $\k$ satisfies  \eqref{KS-k}, 
then $\g=\k^*\k$  is self-adjoint and satisfies \eqref{KS}; 
and, in the opposite direction, if $\g$ is self-adjoint and satisfies  \eqref{KS}, 
then $\g=\k^*\k$, with  $\k$ satisfying \eqref{KS-k}. 

\DETAILS{  \eqref{HE-k} implies  \eqref{HE}, with $\k^*\k=\g$. 
The converse is true, if either $\g_0>0$ or \eqref{HE} has the uniqueness property.}
\end{prop}

\begin{proof} 
Since $\g$ is assumed to be self-adjoint, $\rho_\g$ and therefore $g(\rho_\g)$ are real. 
Under conditions \eqref{SI1}-\eqref{SI2}, or,  more generally,  \eqref{g1g2}-\eqref{g-cond1'}, 
the operator $h_\g$ is self-adjoint and therefore generates a unitary evolution which we denote by $U^\g(t, s)$. 
We write $U^\g_t\equiv U^\g(t, 0)$ and $\al^\g_t(\s):=U^\g_t\s ({U^\g_t})^{-1}$. 
(The evolution $\al^\g_t$ is generated by the linear operator $\s\ra [h_\g, \s]$.) 
The evolution $\al^\g_t$ is differentiable in $t$ on an appropriate dense set 
(e.g., the non-abelian Sobolev space defined in \eqref{LOC1'}) which is preserved by it
and has the following properties:  
\begin{align}\label{gam-int-eq}
&\g(t)  \text{ satisfies  \eqref{KS} with an i.c. }\ \g_0 \iff  \g(t)=\al^\g_t(\g_0 ),\\
\label{gam-k-rel}
&  \al^\g_t(\k^*_0 \k_0 )=\al^\g_t(\k^*_0)\al^\g_t(\k_0)=\al^\g_t(\k_0)^*\al^\g_t(\k_0).
\end{align} 
If $\k(t)$ satisfies \eqref{KS-k} with an initial condition $\k_0$, then $\k(t)=\al^\g_t(\k_0 )$ 
and therefore by \eqref{gam-int-eq} and \eqref{gam-k-rel}, $\g(t)=\k^*(t)\k(t)=\al^\g_t(\k_0)^*\al^\g_t(\k_0)$ satisfies \eqref{KS} with the initial condition $\g_0 =\k^*_0 \k_0$.

On the other hand, if $\g(t)$ satisfies \eqref{KS} with an initial condition $\g_0 =\k^*_0 \k_0$, then, by \eqref{gam-k-rel}, $\g(t)=\k(t)^*\k(t)$, where $ \k(t):=\al^\g_t(\k_0)$ and therefore satisfies $\p_t \k  = i [h_{\g},\k]$ 
  (with $\g=\k^*\k$).
 \DETAILS{The first part is straightforward. To prove the second one, we compute for $\k^*\k=\g$:
\[\p_t \g  - \frac1i [h_{\rho_{\g}},\g]=\mu^*\k+\k^*\mu,\] 
  where $\mu:= \p_t \k  - \frac1i [h_{\rho_{\k^*\k}},\k]$. Now, if $\k_0 >0$, then $\k>0$ and $\mu^*=\mu$, self-adjoint. Let $\{\phi_n\}$ be a complete set of eigenfunctions of $\k$ and let $\{\lam_n\}$ be the corresponding eigenvalues. We have $\lan \phi_m, (\mu \k+\k \mu)\phi_n\ran = (\lam_m+\lam_n)\lan \phi_m, \mu\phi_n\ran$. Hence if  \eqref{HE} holds then   $(\lam_m+\lam_n)\lan \phi_m, \mu\phi_n\ran=0$ and $\lan \phi_m, \mu\phi_n\ran=0$, since $\lam_m, \lam_n>0$.}
 \end{proof}

We designate  LWP, GWP, AC to stand for 
`local well-posedness',
`global well-posedness' and `asymptotic completeness'. 
The latter says that for every $\k_0 \in W^0$, there exists $\k_\infty \in W^0$ such that  the solution to \eqref{KS-k} with the initial condition $\k_0 \in W^0$  satisfies
\begin{align}\label{kapscat}
& {\| \k(t) - e^{i \Delta t} \k_\infty e^{-i \Delta t}\|}_{I^2} \ra 0.
\end{align}

\begin{prop} \label{prop:ka-to-gam} Schematically, we have
\begin{itemize}
\item[(i)]  LWP($\k$) $\Rightarrow $ LWP($\g$);
\item[(ii)]  GWP($\k$) $\Rightarrow $ GWP($\g$);
\item[(iii)]  AC($\k$) $\Rightarrow $ AC($\g$). 
\end{itemize}
\end{prop}

The items (i) and 9ii) follow by $\gamma = \k^\ast \k$ 
using also the uniqueness of trace-class solutions of \eqref{KS}.
Item (iii) follows from the definitions \eqref{SRscat} and \eqref{kapscat} by setting $\gamma_\infty = \k_\infty^\ast\k_\infty$.

\begin{thm}
\label{thm:GWP-ST-ka} 
The equation \eqref{KS-k} with 
\eqref{g1g2}-\eqref{g-cond1'}
and an initial condition $\k_0=\k(t=0)$ satisfying, for some integer $b > d/2$,
\begin{align}\label{loc-dec-data}
\big\| \langle \nabla \rangle^b \k_0 \big\|_{W^0} + \big\|\jx^b \k_0 \big\|_{W^0} \leq B < \infty,
\end{align}
is GWP and AC.
\end{thm}

The proof of Theorem \ref{thm:GWP-ST-ka} is 
given after Theorem \ref{thm:loc-dec-ka} below.
Theorem \ref{thm:GWP-Scatt} follows from Theorem \ref{thm:GWP-ST-ka} and Proposition \ref{prop:ka-to-gam}.

 
\begin{remark}  
  
\begin{itemize}


\item[(1)] Unlike \eqref{KS}, Eq \eqref{KS-k}  has a {\it gauge symmetry}: for any unitary operator $U$, which commutes with $h_{\rho}$ (i.e. is a symmetry of $h_{\rho}$), if $\k$ is a solution to  \eqref{KS-k}, then so is $U\k$. Note that $U\k$ produces the same $\g$ as $\k$: $(U\k)^* U\k=\k^* \k$.
  
\item[(2)] The nonlinearity $\hat g(
\k):=[g(\rho_{\k^*\k}),\k]$ inherits 
  the gauge symmetry  very important for our analysis:
\begin{align}\label{nonlin-gauge} e^{i\chi} \hat g(\k)e^{-i\chi} =\hat g(
e^{i\chi} \k e^{-i\chi} )\end{align} for any differentiable function $\chi$. To see this it suffices to observe that $e^{i\chi} \rho_{\g}e^{-i\chi} = \rho_{\g}=  \rho(e^{i\chi} \g e^{-i\chi} )$.

\item[(3)] (I. Chenn) In the time-dependent case, the following equation  
\begin{align}\label{HE-k'}
   \p_t \k  =i h_{\k^\ast\k}\k
\end{align}
 also leads to \eqref{KS}, if $\g=\k^\ast\k$, however it 
 does not give the time-independent equation corresponding to \eqref{KS}. 
 
\end{itemize}

\end{remark}

\medskip
\subsection{Local decay for $\k$ and main propositions}\label{sec:LD-appr}
At the heart of understanding the long-time behaviour of solutions 
is the {\it local decay} property which shows that, as time progresses, solutions move out of bounded regions of the physical space
and quantifies how quickly this happens. It is usually stated as a bound 
on a local norm, i.e. a norm measuring concentration of the solution in bounded domains. 
If such a bound is sufficiently strong, it implies the global existence and scattering property.
 

To formulate precisely a local decay result for the Hilbert-Schmidt operator $\k$,
with an integral kernel $\k(x,y)$,
we introduce the local norms
\begin{align}\label{kap-loc-norm}
\|\k \|_{L^q_rL^p_c} &\equiv \|\tilde\k \|_{L^q_rL^p_c} :=\|\|\tilde\k \|_{L^p_c} \|_{L^q_r}. 
\end{align} 
for $1\leq p,q\leq \infty$,
where $\tilde\k(r,c)$ is the kernel given by
\begin{align}\label{kappacoord}
\tilde\k(r,c):=\k(x,y), \qquad \mbox{where} \quad r:=y-x, \quad c:=\frac12(y+x).
\end{align}

\begin{definition}
We say that $\k(t)$ satisfies the {\it local decay} property LD$_\nu$($\k$) 
if $\|\k(t) \|_{L^2_rL^s_c} \ls  t^{-\nu}$ for $\nu=d(1/2-1/s)$ and some $s>2$.
\end{definition}

\begin{prop} \label{prop:LD->ST}
With the notation LD$_\nu$ above (in addition to the notation LWP, GWP, AC introduced  
in the paragraph preceding Proposition \ref{prop:ka-to-gam}), we have, if $\nu>1$,
\begin{align}\label{LD->ST}
\mbox{LWP($\k$) $+$ LD$_\nu$($\k$) $\quad \Longrightarrow \quad$ GWP($\k$) $+$ AC($\k$)}.
\end{align} 
\end{prop}
The proof of \eqref{LD->ST} is given in Section \ref{sec:loc-exist}
and relies on standard arguments. 
Proposition \ref{prop:LD->ST} reduces the proof of Theorem \ref{thm:GWP-ST-ka} 
to the proof of LD$_\nu$($\k$)
(LWP is standard and given by Theorem \ref{pr:loc1}) to which we proceed. 

In what follows, the exponent $\al$ appearing in several interpolation type inequalities
is always assumed to satisfy the condition 
\begin{align}\label{al-cond}
\al \ge 0 \quad \text{and \quad $\al<1$ for $d$ even, \quad $ \al\le1$ for $d$ odd.} 
\end{align}

Here is the key local decay result for Eq \eqref{KS-k}: 

\begin{thm}[Local decay] \label{thm:loc-dec-ka}
Assume $d\le 3$ and conditions 
\eqref{g1g2}-\eqref{g-cond1'}.
Let $\k$ be the local-in-time solution of \eqref{KS-k} (see Theorem \ref{pr:loc1}), 
with initial datum $\k(t=0) =: \k_0$ satisfying \eqref{loc-dec-data}.
Then there exists $\lambda_0=\lambda_0(B)>0$ small enough such that for all $|\lambda_1|,|\lambda_2| \leq \lambda_0$
we have
\begin{align}\label{loc-dec-ka}
\|\k \|_{L^2_rL^s_c} 
&\ls |t|^{-\al b} \|\jx^b\k_0 \|_{W^0}^\al \| \k_0 \|_{W^0}^{1-\al},
\\
\label{loc-dec-ind}
& \al b=d(\frac12-\frac1s), \quad 2 \le s \leq \infty.
\end{align}
\end{thm}  



\begin{proof}[Proof of Theorem \ref{thm:GWP-ST-ka}] 
A local-in-time solution to \eqref{KS-k}, under the conditions stated in Theorem \ref{thm:loc-dec-ka}, 
is provided by the local existence Theorem \ref{pr:loc1}.

The global existence and scattering for \eqref{KS-k} 
follow from the local existence Theorem \ref{pr:loc1},  the local decay from Theorem \ref{thm:loc-dec-ka} 
and Proposition \ref{prop:LD->ST}.
\end{proof}

The proof of Theorem \ref{thm:loc-dec-ka} will follow from 
a non-commutative version of a weighted Gagliardo-Nirenberg-Klainerman-type interpolation inequality
(Proposition \ref{prop:tildekap-est}) and an a priori energy estimate in the weighted space $W^b$ 
(Proposition \ref{prop:AprioriBnds}).

\begin{prop}\label{prop:tildekap-est} 
For any $s\ge 2$ 
 and $\al \in[0,1)$ for $d$ even, $\al \in [0,1]$ for $d$ odd, we have
\begin{align}\label{L2Ls-est} 
\|\k(t) \|_{L^2_rL^s_c}
  & \ls |t|^{-\al b} \|\k(t) \|_{W^b}^\al \| \k(t) \|_{W^0}^{1-\al},\qquad \al b=d(\frac12-\frac1s).
\end{align}
\end{prop}

This proposition is proven in Section \ref{sec:Pf-tildekap-est}. The next result gives a priori energy inequalities.
\DETAILS{The main idea here is first to extend the Gagliardo-Nirenberg inequality to non-abelian spaces with the derivations replaces by the commutators 
\begin{align}\label{D}
D_\ell\k & := [\partial_{x_\ell} ,\k], \qquad D =(D_1,\dots,D_d) = [\n,\,\cdot\,], \end{align}
 and then observe that 
$ J=-2it  U^*DU$, where $J$ is the commutator defined in  \eqref{J'} and  $U: \psi(x)\ra  e^{- i x^2/4t}\psi(x)$.}

\begin{prop}[A priori bounds] \label{prop:AprioriBnds}
 Assume  $d\le 3$. 
 Let $b> d/2$ be an integer, 
 and $\lambda_1,\lambda_2$ denote the coupling constants in \eqref{g1g2}. 
There exists an absolute constant $c_0$ such that,
any solution $\k$ to equation \eqref{KS-k} which satisfies for some time $s\geq 10$ 
\begin{align}\label{eps-cond} 
|\lambda_1| 
  \| \k(s) \|_{W^b}^2  +  |\lambda_2| \| \k(s) \|_{W^b}^{2\beta} 
  	< \frac{c_0}{2^{\max\{3,2\beta+1\}}}, 
\end{align}
also satisfies for any $t>s$ 
\begin{align}\label{Jb-est}
& \| \k(t)\|_{W^b} \leq 2 \| \k(s)\|_{W^b}. 
\end{align}  
 
\end{prop}


\DETAILS{\begin{remark}
Proposition \ref{prop:AprioriBnds} and Theorem \ref{thm:loc-dec-ka} {\bf can be extended} to non-integer $b$'s
by using the formula
\[ J^b_\ell \k = (-i\p_s)^b\Big|_{s=0} e^{is J_\ell} 
= (-i\p_s)^b\Big|_{s=0} e^{is  j_\ell} \kappa e^{-is  j_\ell} \]
\end{remark}}

\begin{remark}
We do not need a smallness condition $|\lambda_1|,|\lambda_2| \ll 1$ in Theorem \ref{thm:loc-dec-ka} 
if we start at a sufficiently large time $t_0$.
We can then solve the final state problem in our setting without assuming weakly nonlinear interactions.
\end{remark}

Proposition \ref{prop:AprioriBnds} is proven in Section \ref{sec:ener-ineq-pf}. 
The main idea here is to use that the Galilean boost observable $J$ is almost conserved under \eqref{KS-k},
see Proposition \ref{prop:jk-evol}.
We remark that the {\it gauge invariance} \eqref{nonlin-gauge} of the nonlinearity, 
and more precisely the invariance of \eqref{KS-k} under the Galilean transformations \eqref{gal-boost}
plays an important role in this proof.



\DETAILS{\begin{proof}[Proof of Theorem \ref{thm:GWP-Scatt}]
The global existence of a solution $\k$ for \eqref{KS-k} 
with initial datum $\k_0$, follows from the local existence Theorem \ref{pr:loc1}, 
the local decay from Theorem \ref{thm:loc-dec} and Proposition \ref{prop:LD->ST}.
Then, the global existence for \eqref{KS} with initial datum $\gamma_0 = \k^\ast_0\k_0$, follows immediately from $\g = \k^\ast \k$ and the uniqueness of solutions for \eqref{KS}.

The scattering property for $\g$ in $I^1$, as stated in \eqref{SRscat}, 
follows directly from the scattering of $\k$ in $I^2$, see \eqref{kapscat}, 
by letting $\g_\infty=\k^\ast_\infty \k_\infty$.
\end{proof}}

The following lemma, which is used in the proof of the local decay for $\g$ and of Lemma \ref{lem:energy-est}, 
establishes a relation between local $rc$- and $xy$-norms: 
 
\begin{lemma}\label{lem:xyrc}
For all $s\geq 2$ we have
\begin{align}\label{xyrc}
{\| \k \|}_{L^s_x L^2_y} \leq {\| \k \|}_{L^2_rL^s_c},
\ \text{ where } {\| \k \|}_{L^s_x L^2_y} := {\| {\| \k(x,y) \|}_{L^2_y}\|}_{L^s_x}.
\end{align}
\end{lemma}

\begin{proof}
Recall the notation \eqref{kappacoord} and introduce the function 
\begin{align}\label{ka-f-transf}
f(x) := {\| \k \|}_{L^2_y}^2(x)= \int |\k(x,y)|^2 dy 
  = \int |\tilde\k(r, x-r/2)|^2 dr. 
\end{align}
Then
${\| \k \|}_{L^s_x L^2_y}^2 = {\| f \|}_{L^{s/2}} \leq \int {\big\| |\tilde\k(r, x-r/2)|^2 \big\|}_{L^{s/2}_x} dr
   = \int {\big\| \tilde\k(r, \cdot) \big\|}_{L^{s}_c}^2 dr
   = {\| \tilde\k \|}_{L^2_rL^s_c}^2$.
\end{proof}

\noindent
{\bf Local decay for $\g$}.
We end this section by discussing local decay for $\gamma=\k^\ast\k$, which is of independent interest,
although it is not used in the proof of Theorem \ref{thm:GWP-Scatt}.
 
\begin{definition}[Local norms for $\g$]
We define the local norm of operators $\g$ through their integral kernels $\gamma(x,y)$ as
\begin{align}\label{gam-loc-norm}
\|\g\|_{(s)}:=\left(\int_{\R^d\times\R^d} |\gamma(x,y)|^s \, dxdy\right)^{1/s},
\end{align}
with the standard adjustment for $s=\infty$.
\end{definition}
Note that
\begin{align} \label{LDk->LDg}
{\| \gamma \|}_{(s)} \leq {\| \k \|}_{L^s_x L^2_y}^2, 
\end{align}
Now, \eqref{LDk->LDg}, \eqref{xyrc}, and Theorem \ref{thm:loc-dec-ka} 
imply that the solutions of \eqref{KS} described in Theorem \ref{thm:GWP-Scatt} have the following local decay property: 
for all $t\in\R$,
\begin{align}\label{loc-dec'}
&\|\g(t) \|_{(s)} \ls 
|t|^{-d(1-2/s)}, \qquad 2 \le s \le \infty .
\end{align}


\DETAILS{Finally, to obtain the local decay property \eqref{loc-dec'} for $\g = \k^\ast\k$, observe that 
\begin{align*}
{\| \g \|}_{(s)} \leq {\| \k \|}_{L^s_xL^2_y}^{2} 
  \leq {\| \k \|}_{L^2_r L^s_c }^{2} \lesssim |t|^{-d(1-2/s)},
\end{align*} 
having used \eqref{loc-dec-ka}, 
with the implicit constant depending on the initial datum \eqref{loc-dec-data}.}

\DETAILS{
Finally, to obtain the local decay property \eqref{loc-dec'} for $\g = \k^\ast\k$, observe that 
\begin{align*}
{\| \g \|}_{(2)} = {\| \k \|}_{I^2}^2, \qquad {\| \g \|}_{(\infty)} \leq {\| \k \|}_{L^2_r L^\infty_c}^2,
\end{align*} 
so that interpolation and \eqref{loc-dec} in Theorem \ref{thm:loc-dec} give
\begin{align*}
{\| \g \|}_{(s)} \leq {\| \g \|}_{(2)}^{2/s} {\| \g \|}_{(\infty)}^{1-2/s} 
\lesssim {\| \k \|}_{I^2}^{4/s} {\| \k \|}_{L^2_r L^\infty_c }^{2(1-2/s)} \lesssim |t|^{-d(1-2/s)},
\end{align*} 
where the implicit constant depend on the initial datum \eqref{loc-dec-data}.}

\bigskip
\section{Proof of Proposition \ref{prop:tildekap-est}}\label{sec:Pf-tildekap-est} 
The main idea here is first to extend the Gagliardo-Nirenberg inequality to the non-abelian Sobolev-type spaces
\begin{align}\label{LOC1'} 
V^s &:= \big\{ \k \in I^2\, : \, \sum_{|\alpha| \leq s} \| D^\alpha \k  \|_{I^2} < \infty \big\}, \end{align}
for any 
 $s\geq0$,  with the grading provided 
 by the commutators 
\begin{align}\label{D}
D_\ell\k & := [\partial_{x_\ell} ,\k], \qquad D =(D_1,\dots,D_d) = [\n,\,\cdot\,], \end{align}
 and then observe that 
the commutator vector-field defined in  \eqref{J'} is related to $D$ by the formula
\begin{align}\label{J-D} J_t=-2it  U_t^*DU_t,\end{align} where  $U_t: \psi(x)\ra  e^{- i x^2/4t}\psi(x)$.

Another important observation used in the proof is that, if we denote the map from operators $\k$ to their transformed kernes $\tilde\k(r,c)$ by $\phi$,  then we have
  \begin{align}\label{D-dc} 
& \phi(D_i\k)=\p_{c_i}  \phi(\k).
\end{align}

Passing from operators $\ka$ to the integral kernels $\tilde\k(r,c):=\k(x,y),$ where $ r:=y-x, \ c:=\frac12(y+x)$ 
(see \eqref{kappacoord}) and applying  the standard Gagliardo-Nirenberg interpolation inequality in the $c$-variable, 
we find
\begin{align}\label{GNI'} 
&\| \tilde\k\|_{L^s_c} \ls \|\p^b \tilde\k\|_{L^2_c}^{\al} \| \tilde\k\|_{L^2_c}^{1-\al},
\\ 
\label{GNI-ind'} 
&b\al=  d(\frac{1}{2}-\frac{1}{s}),\ s\ge 2,  
\end{align} 
 and $0\le \al<1$ for $d$ even and $0\le \al\le1$ for $d$ odd. 
 Applying to this the H\"older inequality with the exponents $1/\al$ and $1/(1-\al)$, we obtain furthermore
 \begin{align}\label{GNI''} 
&\|\tilde\k\|_{L^2_rL^s_c} \ls \|\p^b \tilde\k\|_{L^2_rL^2_c}^{\al} \| \tilde\k\|_{L^2_rL^2_c}^{1-\al}.
\end{align} 
 Next, we claim that 
  \begin{align}\label{L2-Wb-norms} 
& \|\p^b \tilde\k\|_{L^2_rL^2_c}= \| D^b\k\|_{I^2}.
\end{align} 
Indeed, we use \eqref{D-dc} to define $D_i^r\k$, for arbitrary positive powers of derivatives, 
by $\phi(D_i^r\k):=\p_{c_i}^r \phi(\k)$,
where $\p_{c_i}^r$ is the standard fractional derivative, see for example \cite{SteinBook1}.
Since $\| \tilde\k\|_{L^2_rL^2_c}= \| \k\|_{I^2}$, \eqref{L2-Wb-norms} follows.
Relations \eqref{GNI''}, \eqref{L2-Wb-norms} and \eqref{D-dc} imply 
\begin{align}\label{GNI-non-ab} 
\|\k(t) \|_{L^2_rL^s_c}
  & \ls  \|D^b\k(t) \|_{I^2}^\al \| \k(t) \|_{I^2}^{1-\al},  
\end{align}
with \eqref{GNI-ind'}, which gives  the non-abelian Gagliardo-Nirenberg inequality.
Now, using \eqref{J-D} and the relation $I^2=W^0$, we convert this into
\begin{align}\label{L2Ls-est'} 
\|\k(t) \|_{L^2_rL^s_c}
  & \ls |t|^{-\al b} \|J^b\k(t) \|_{W^0}^\al \| \k(t) \|_{W^0}^{1-\al}, 
\end{align}
 with \eqref{GNI-ind'}, $b$ positive integer and  $\al \in[0,1]$ satisfying \eqref{al-cond}, 
which is a stronger (scale-invariant) version of \eqref{L2Ls-est}. $\hfill \Box$

\bigskip
\section{Local norm and density estimates}\label{sec:den-est}   

\DETAILS{To prove Proposition \ref{prop:AprioriBnds}, 
we will use propagation estimates based on the Galilean boost observable $J$. 
To this end, we show that $J$ is almost conserved under \eqref{HE-k}: 

\begin{prop}[Propagation of $J$]\label{prop:jk-evol} 
Denote $D_\g\k:= i \p_t \k - [h_{\g}, \k]$. Then
 \begin{align}\label{Jk-evol}
D_\g J\k  & =JD_\g \k  + [v*\rho_{J\g},\k]
 \end{align} \end{prop} 
A proof of this proposition is given in Appendix \ref{sec:Jk-evol-pf}. 
The {\it gauge invariance} \eqref{nonlin-gauge} of the nonlinearity plays an important role in this proof.
}

As a preparation for demonstrating Proposition \ref{prop:AprioriBnds}, we prove several inequalities on local norms and densities.

\begin{lemma}[Estimates on the density] \label{lem:rho-est} 
Let $1/w+1/w'=1/q$. Then 
\begin{align}\label{rho-est'}\|\rho(\k \k')\|_q &\ls\|\k \|_{ L^{2}_r L^{w}_c} \|{\k'}\|_{L^{2}_r L^{w'}_c}.\end{align}
  \end{lemma} 
  
\begin{proof} 
Using that $\rho(\k \k') = \int\k(x, y) \k'(y, x) dy$ and passing from $\k$ to $\tilde\k$, 
we find 
\begin{align}
\begin{split}
\label{vrho-est}\|\rho(\k \k')\|_q 
& = {\Big\| \int\tilde\k(y-x, \frac12(x+y)) \tilde{\k'}(x-y, \frac12(x+y)) dy \Big\|}_{L^{q}_x} 
\\
& = {\Big\| \int\tilde\k(-r, x-\frac12 r) \tilde{\k'}(r, x-\frac12 r) dr \Big\|}_{L^{q}_x}  
\\
&\leq \int {\big\| \tilde\k(-r, x-\frac12 r) \tilde{\k'}(r, x-\frac12 r) \big\|}_{L^{q}_x} dr   
\\
&=\int\|\tilde\k(-r, c) \tilde{\k'}(r, c)\|_{L^{q}_c} dr
\\
&\ls \int\|\tilde\k(-r) \|_{L^{w}_c} \|\tilde{\k'}(r)\|_{L^{w'}_c} dr.
\end{split}
\end{align} 
Upon application of the Schwarz inequality, this yields \eqref{rho-est'}. \end{proof}

Now, applying \eqref{L2Ls-est} gives 

\begin{cor} \label{cor:rho-est} 
Let $q\ge 1$,   $\al, \al' \in[0,1]$ satisfy \eqref{al-cond} 
and $\nu:=\al b+\al' b'=d(1-\frac1q)$. 
Then 
\begin{align}\label{rho-est}
\|\rho(\k \k')\|_q 
  &\ls |t|^{-\nu} \|\k \|_{W^b}^\al \| \k \|_{W^0}^{1-\al}\|\k' \|_{W^{b'}}^{\al'} \| \k' \|_{W^0}^{1-\al'}.
\end{align} 
\end{cor}

Next, we have

\begin{lemma}[Products of functions and half-densities] \label{lem:fkap-est} Let 
$f$ be a multiplication operator  by $f\in L^p$. Then 
\begin{align}\label{fkap-est0}
\|f\k \|_{W^0} 
&\ls \|f\|_{L^{p}} \|\k\|_{L^2_rL^{s}_c}
\\
\label{fkap-est}
&\ls |t|^{-\al b}\|f\|_p  \|J^b\k \|_{W^0}^\al \| \k \|_{W^0}^{1-\al},
\end{align} 
where $\al b= d/p$ 
 and $1/p+1/s=1/2$. 
\end{lemma} 
\begin{proof}
Let $p^{-1}+ {s}^{-1}= \frac12$. We estimate 
\begin{align*}
\|f\k \|_{W^0}^2 &= \iint |f(x)\k(x, y)|^2 dxd y \notag
\\ 
&= \iint |f(c+\frac12 r)\tilde\k(r, c)|^2 dr dc \notag
\\
&\le \int \|f\|_{L^{p}}^2 \|\tilde\k\|_{L^{s}_c}^2 dr = \|f\|_{L^{p}}^2 \|\tilde\k\|_{L^2_rL^{s}_c}^2.
\end{align*}
This gives \eqref{fkap-est0}. The latter and \eqref{L2Ls-est} imply \eqref{fkap-est}. 
\end{proof}

Next, we prove the following elementary inequality 
\begin{align}\label{rhoJ-est1} \left| 	\rho_{J (\k^\ast \k)} (x) \right|^2 
& \le  	  2\rho_{J\k^\ast J\k} (x) \rho_{\k^\ast \k} (x).
\end{align}
To prove this, we use $J (\k^\ast \k)= (J\k^\ast) \k +\k^\ast (J\k)$ to estimate\\
\begin{align}
\notag& \left| 	\rho_{J (\k^\ast \k)} (x) \right|^2 
 \le  	2(\int_{\R^3} \left| \overline{J\k(z,x)} \k(z,x) \right| dz)^2 \, 
\\
\notag & \le  
 	 2\int_{\R^3} |J\k(z,x)|^2 dz  \int_{\R^3} |\k(z,x)|^2 dz 
\end{align}
which implies the inequality  \eqref{rhoJ-est1}.

%


\bigskip
\section{Approximate Galilean conservation law}\label{sec:approx-conservq}

In this section we prove energy-type inequalities for `half-densities' $\ka$. The first lemma is related to the invariance of \eqref{KS} and \eqref{KS-k} 
under Galilean transformations \eqref{gal-boost}. In what  follows, we use the following relation (which we call Jacobi-Leibniz rule)
\begin{align}\label{JacobiJ}
J[a, b]= [J a, b]+[a, J b],
\end{align}  
which follows from the Jacobi identity $[[A, B], C]+ [[C, A], B]+[[B, C], A]=0$.
\begin{prop}[Galilean invariance] \label{prop:jk-evol} 
Denote $D_\g\k:= \p_t \k  - i[h_{\g}, \k]$. 
Then $D_\g$ and $J$ almost commute in the sense that
\begin{align}\label{Jk-evol}
D_\g J\k  & =JD_\g \k  + i[d  g (\rho_{\g}) \rho_{J\g},\k].
\end{align}
Moreover, if we let $J^2 = J_{\ell_2}J_{\ell_1}$, for any $\ell_1,\ell_2 = 1,\dots d$, then we have
\begin{align}\label{J2k-evol}
\begin{split}
D_\g J^2\k & = J^2 D_\g \k + i[d g (\rho_{\g}) \rho_{J^2\g}, \k]
  \\
  & + i [d^2  g (\rho_{\g}) \rho_{J_{\ell_1}\g} \rho_{J_{\ell_2}\g}, \k]
  + i [d g (\rho_{\g}) \rho_{J_{\ell_2}\g}, J_{\ell_1} \k]
  + i [d g (\rho_{\g}) \rho_{J_{\ell_1}\g}, J_{\ell_2} \k],
\end{split}
\end{align} 
where, recall,  $d^k g$ is the $k$-th  G\^ateaux derivative of $g$. 
\end{prop} 
 
\begin{proof} 
First, we compute 
\begin{align}\label{J-inter0}
[j,  \p_t \k] = \p_t  [j, \k] + 2[p, \k].
\end{align}
This, together with \eqref{JacobiJ}, implies 
\begin{align}\label{J-inter1}
[j, [h_{0},\k]] = [[j, h_{0}],\k]+ [h_{0}, [j, \k]] =[h_{0}, [j, \k]] - 2i [p, \k]. 
\end{align}
Subtracting \eqref{J-inter1} times $i$ from \eqref{J-inter0} 
we obtain $[j, \p_t \k - i[h_{0},\k]] = \p_t  [j, \k] - i[h_{0}, [j, \k]]$,
which can be rewritten as
 \begin{align}\label{J-free-evol}
JD_0 & = D_0 J, \qquad D_0\k := \p_t \k - i[h_{0}, \k].
\end{align} 

To deal with the difference $(D_\g - D_0)\ka = -i[g(\rho_{\g}),\k]$
we use that $g$ is covariant under translations and gauge transformations and therefore it is also covariant under the 
Galilean transformations \eqref{gal-boost}.

For a general nonlinearity $f(\g)$, the covariance relation states $U_{v, t} f(\g)U_{v, t}^{*} =f(U_{v, t} \g U_{v, t}^{*})$. 
Differentiating it with respect to $v$ at $v=0$, we find
\begin{align}\label{nonlin-j}
[j, f(\g)] =d f(\g) [j, \g].
\end{align}
Taking here $f(\g)=g(\rho_{\g})$ and using that $d f(\g)\xi= d g(\rho_{\g})d \rho_{\g}\xi$ and $d \rho_{\g}\xi= \rho_{\xi}$, 
this gives
\begin{align}\label{J-inter'} 
& [j, g(\rho_{\g})]=d  g (\rho_{\g}) \rho_{[j, \g]},
 \end{align} 
which, together with the Jacobi-Leibnitz identity \eqref{JacobiJ}, yields
\begin{align}\label{J-inter''}
J[ g(\rho_{\g}), \k] = [J g(\rho_{\g}), \k] + [ g(\rho_{\g}), J\k]. 
\end{align}
We 
combine \eqref{J-free-evol} and \eqref{J-inter'}-\eqref{J-inter''} to obtain  
\begin{align}\label{Jk-evol'}
& J D_\g \k = D_\g J\k - i [d  g (\rho_{\g}) \rho_{J\g},\k].
\end{align}
which is \eqref{Jk-evol}.

To prove \eqref{J2k-evol}, we recall $J^2 = J_{\ell_2}J_{\ell_1}$ and iterate \eqref{Jk-evol'}. 
First, we have
\begin{align}\label{J2k-evol1} J^2 D_\g\k  & = J_{\ell_2}D_\g J_{\ell_1} \k 
  - i J_{\ell_2} [d  g (\rho_{\g}) \rho_{J_{\ell_1}\g}, \k].
\end{align}
Then, using \eqref{Jk-evol'} again, we find for the first term on the right-hand side
\begin{align}\label{J2k-evol2}
J_{\ell_2}D_\g J_{\ell_1} \k& = D_\g J^2 \k -  i [d g (\rho_{\g}) \rho_{J_{\ell_2}\g}, J_{\ell_1} \k].
\end{align} 
For the second term on the right-hand side of \eqref{J2k-evol1}, we use relation \eqref{JacobiJ}
to find
\begin{align}\label{J2k-evol3}
\notag
J_{\ell_2} [d g (\rho_{\g}) \rho_{J_{\ell_1} \g}, \k] & =  [d^2  g (\rho_{\g}) \rho_{J_{\ell_2}\g} \rho_{J_{\ell_1}\g}
 + d  g (\rho_{\g}) \rho_{J_{\ell_2}J_{\ell_1}\g}, \k]
 \\
& + [d  g (\rho_{\g}) \rho_{J_{\ell_1}\g}, J_{\ell_2}\k] .
\end{align}  
Combining \eqref{J2k-evol1}, \eqref{J2k-evol2} and \eqref{J2k-evol3} gives  \eqref{J2k-evol}.
\end{proof}

\medskip
With $J$ as defined in \eqref{J'}, let us denote
\begin{align}\label{J^mnotation}
J^m = \prod_{i=1}^d J_i^{m_i}, \qquad m=(m_1,\dots,m_d).  
\end{align}

Using the key relation \eqref{Jk-evol} we derive the following identity for the evolution of the weighted energy:

\begin{lemma}[Evolution of the weighted energy]\label{prop:energy-relat} 
Assume $\k$ satisfies \eqref{KS-k}. Then
\begin{align}\label{energy-relat}
&\frac{1}{2} \p_t\|J^m\k \|_{W^0}^2= \im \lan J^m \k,  R_m \ran_{W^0},
\end{align}
with
\begin{align}
\label{energy-relat-R} 
\begin{split}
& R_m := \sum_{k=1}^{|m|} \sum_a
  c_{k,m} [d^k  g (\rho_{\g})\prod_{i=1}^k \rho_{J^{s_i}\g}, J^a \k],
  \\ & \sum_{i=1}^k s_i+|a|=|m|, \quad s_i>0, \quad (|a|\le |m|-1),
\end{split}
\end{align}
for some constants $c_{k,m}$.
Here $\lan \k,\k' \ran_{W^0}$ is the $W^0=I^2$ inner product defined in \eqref{HSinner}.
\end{lemma}

\begin{proof} 
For simplicity, we show \eqref{energy-relat}--\eqref{energy-relat-R} only in the $|m|=1, 2$ cases,
which is sufficient to do our a priori estimates in dimensions $d\leq 3$.
It will be clear to the reader how this generalizes applying the arguments below and Fa\'a-di Bruno's formula.

We compute $\p_t\|J^m\k \|_{W^0}^2= 2\re\lan J^m \k,  \p_tJ^m \k\ran_{W^0}$. 
Now, using $  \p_t \k'   = D_\g\k' + i [h_{\g}, \k']$ with $\k'=J^m \k$ 
and $\re\lan J^m \k,  i [h_{\g},J^m\k]\ran_{W^0}=0$ yields
\begin{align}\label{Jm-energy-eq1}
\p_t\|J^m\k \|_{W^0}^2= 2\re \lan J^m \k,  D_\g J^m\k \ran_{W^0}.
\end{align}
Now, letting $|m|=1$ and applying \eqref{Jk-evol} with $\g=\k^*\k$ and $D_\g\k=0$ (by \eqref{KS-k}) gives  
\begin{align}\label{Jenergy-est1}
\frac{1}{2} \p_t\|J_\ell\k \|_{W^0}^2 = \re \lan J_\ell \k,  
 i [d  g (\rho_{\g}) \rho_{J_\ell\g},\k]\ran_{W^0} ,  \end{align}
which gives \eqref{energy-relat}--\eqref{energy-relat-R} with $|m|=1$.

To compute in the case $|m|=2$, we begin with \eqref{Jm-energy-eq1} with $|m|=2$
and simplify our notation by denoting $J^2 = J_{\ell_2}J_{\ell_1}$, for any $\ell_1,\ell_2 = 1,\dots d$.
To compute the right-hand side of \eqref{Jm-energy-eq1}, we plug \eqref{J2k-evol} into  \eqref{Jm-energy-eq1} with $|m|=2$  and $\g=\k^*\k$, and using $D_\g\k=0$ (by \eqref{KS-k}) to obtain 
\begin{align*}
\frac{1}{2} \p_t\|J^2\k \|_{W^0}^2 =\im \lan J^2 \k,    [d  g (\rho_{\g}) \rho_{J^2\g}, \k]
 + [d^2 g (\rho_{\g}) \rho_{J_{\ell_1}\g} \rho_{J_{\ell_2}\g}, \k] 
 \\ + [d g (\rho_{\g}) \rho_{J_{\ell_2}\g}, J_{\ell_1}\k]
 +  [d g (\rho_{\g}) \rho_{J_{\ell_1}\g}, J_{\ell_2}\k] \ran_{W^0},
\end{align*}
which is of the form \eqref{energy-relat}--\eqref{energy-relat-R} with $|m|=2$. 
\end{proof}

\bigskip
\section{Proof of Proposition \ref{prop:AprioriBnds}}\label{sec:ener-ineq-pf}
This is our main lemma on the control of the evolution of the weighted energy:

\begin{lemma} \label{lem:energy-est} 
Assume \eqref{g1g2}-\eqref{g-cond1'} and $d\le 3$ and let $\k(t)$ satisfy \eqref{KS-k} with a self-interaction $g$ as in \eqref{g1g2}. 
Then, for $b = [d/2]+1$, there exists an absolute constant $C_0$ such that
\begin{align}\label{energy-est}
\begin{split}
\Big| \frac{d}{dt} \|\k(t) \|_{W^b} \Big| 
	\leq C_0 |\lambda_1| \, |t|^{-d(1+\frac1p -\frac1q)} \| \k(t)\|_{W^0}\| \k(t) \|_{W^b}^{2}
\\ 
+ C_0 |\lambda_2| \, |t|^{- d\beta} \| \k(t) \|_{W^b}^{2\beta+1}, 
\end{split}
\end{align} 
where $(p, q)$ is an admissible pair from conditions \eqref{g-cond2}--\eqref{pqcond1}, 
and $\beta$ is the exponent in \eqref{g-condbeta}. 
\end{lemma} 

This statement holds in $d \geq 4$ as well by appropriately modifying the assumptions \eqref{g-cond0}-\eqref{g-condbeta},
see Remarks \ref{rem:d>3} and \ref{rem:d>3'}.
The constant $C_0$ appearing in \eqref{energy-est} determines the constant $c_0$ in \eqref{eps-cond}.
Before proving Lemma \ref{lem:energy-est} let us show how it implies the main Proposition \ref{prop:AprioriBnds}.

\begin{proof}[Proof of Propositions \ref{prop:AprioriBnds}] 
Integrating inequality \eqref{energy-est} and using that $\rho_1 := d(1+\frac1p -\frac1q)-1>0$,
by \eqref{pqcond}, 
and $\rho_2 := d\beta - 1 > 0$ by \eqref{g-condbeta}, we obtain 
\begin{align}\label{kap-t-est}
& \Big| \| \k(t)\|_{W^b} - \| \k(s)\|_{W^b} \Big| \le 
	C_0' \sup_{s\le r\le t} \big( |\lambda_1| \| \k(r)\|_{W^0} \| \k(r) \|_{W^b}^2 +
	|\lambda_2| \| \k(r) \|_{W^b}^{2\beta+1} \big),
\end{align} 
where $C_0' := C_0 \min(\rho_1,\rho_2)^{-1}$.
Letting $A(t):=\sup_{s\le r\le t} \| \k(r) \|_{W^b}$, \eqref{kap-t-est} gives 
\begin{align}\label{Apr0} 
A(t) \le A(s) +  C_0' \big( |\lambda_1|A(t)^3 + |\lambda_2|A(t)^{2\beta+1} \big). 
\end{align}

For fixed time $s$ as in the statement, see \eqref{eps-cond}, let us consider the set of times
\begin{align}\label{Apr1}
\mathcal{T} := \{ t  \in [s,\infty) \, : \, A(t) \leq 2A(s) \}.
\end{align}
$\mathcal{T}$ is non-empty and closed by definition,
since from Theorem \ref{pr:loc1}(iii) we know that $A(t)$ is a continuous function (for proper solutions $\k$).
Moreover, if $s<t \in \mathcal{T}$, then from \eqref{Apr0} and assumption \eqref{eps-cond} with $c_0 = (C_0')^{-1}$ we get
\begin{align*}
A(t) \le A(s) +  2^{\max\{3,2\beta +1\}}
  C_0' \big( |\lambda_1|A(s)^3 + |\lambda_2|A(s)^{2\beta+1} \big) < 2A(s).
\end{align*}
Thus, by continuity, there exists $\delta >0$ such that $A(t') \leq 2A(s)$ for $|t' -t| < \delta$.
It follows that $\mathcal{T}$ is open and therefore $\mathcal{T} = [s,\infty)$, which is 
the desired statement \eqref{Jb-est}.
\end{proof} 
 
 

\begin{proof}[Proof of Lemma \ref{lem:energy-est}] 
The starting point is \eqref{energy-relat}.
To estimate the right-hand side of \eqref{energy-relat} for $|m|=1,2$,
\DETAILS{In this case, we have only three types of terms 
\begin{align}
\begin{split}
& \lan J^m \k, [d^k g (\rho_{\g}) (\rho_{J^s\g})^k, J^a\k]\ran_{W^0}, 
\\
& 1\le k\le |m|, \quad |s| \geq 1, \quad |s|k+|a|=|m|, \quad |m|= 1, 2.
\end{split}
\end{align} 
where we denote with $d^2 g(\rho_{\g}) (\rho_{J^s\g})^2$
any of the expressions $d^2 g(\rho_{\g}) [\rho_{J_{\ell_1}\g}, (\rho_{J_{\ell_2}\g}]$ for $\ell_1,\ell_2=1,\dots d$.
By} 
we use the non-commutative Schwarz inequality to obtain 
\begin{align}\label{RHS-est}
\begin{split}
& |\lan J^m \k,  
  [d^k g(\rho_{\g}) (\rho_{J^s\g})^k, J^a\k]\ran_{W^0}|
  \\ & \ls  \| J^m\k\|_{W^0} \|[d^k g(\rho_{\g}) (\rho_{J^s\g})^k, J^a\k] \|_{W^0}, 
\end{split}
\end{align} 
where
\begin{align}\label{RHS-estp}
1\le k\le |m|, \quad |s| \geq 1, \quad |s|k+|a|=|m|, \quad |m|= 1, 2. 
\end{align}
Now, we claim the following estimates: for parameters as in \eqref{RHS-estp} we have
\begin{align}\label{RHS-est'}
\begin{split}
& \|[d^k g_i(\rho_{\g}) (\rho_{J^s\g})^k, J^a\k] \|_{W^0} \ls t^{-\nu_i} \| \k\|_{W^m}^{p_i}, \qquad i=1,2,
\end{split}
\end{align}
where, recall, $g_i$ are the components of $g$ in \eqref{g1g2} satisfying \eqref{pqcond} and
\begin{align}\label{RHS-estp'}
\begin{split}
& \nu_1 = d(1+1/p-1/q), \qquad p_1 = 3,
\\
& \nu_2 = \beta d, \qquad p_2 = 2\beta + 1.
\end{split}
\end{align}

\medskip
{\it Estimates for $g_1$}.
We begin with $k=1$.  \eqref{fkap-est} and \eqref{rho-est} (with $\al=\al'=1$), 
together with 
  the relation $\rho_{J(\k^*\k)}=\rho_{J(\k^*)\k} +\rho_{\k^*J(\k)} $ and assumption  \eqref{g-cond1}, 
  give 
\begin{align}
\label{vrho-kap-est'} 
\begin{split}
  \|[d g_1 (\rho_{\g}) \rho_{J^s\g}, J^a\k]\|_{W^0}
  &\ls t^{-b}\|d g_1 (\rho_{\g}) \rho_{J^s\g}\|_p  \|J^{b+a}\k \|_{W^0}
\\
&\ls t^{-  b}\| \rho_{J^s\g}\|_q  \|\k \|_{W^{b+a}}
\\
&\ls  t^{-b-\nu} \|\k \|_{W^{b'+s'}} \|\k \|_{W^{b''+s''}} \|\k \|_{W^{b+a}},
\end{split}
\end{align}
where $s'+s''=|s|,$ $\nu:=d(1-\frac1q)$, $b= d/p$, and $b'$ and $b''$ 
are any non-negative numbers satisfying $b'+b''=d(1-\frac1q)$. 
Since $d(1-\frac1q)\le |m|$ and $|s|+|a|= |m|$, we can choose $s', s'',b',b'', a$ so that $b'+s', b''+s'', b+|a|\le |m|$.
Since $\nu+b=d(1+1/p-1/q)$,  this gives
\begin{align}\label{vrho-kap-est}
& \|[d  g_1 (\rho_{\g}) \rho_{J^s\g}, J^a\k]\|_{W^0}\ls t^{-d(1+\frac1p -\frac1q)}  \|\k \|_{W^m}^3,
\\
& d(1-\frac1q)\le |m|, 
\quad |s|+|a|= |m|.
\end{align}
The latter conditions imply that $d(1+1/p-1/q)+|a|\le 2 |m|$. 
Since $d(1+1/p-1/q)>1$ 
this gives $1+|a| < 2 |m|$. Equation \eqref{vrho-kap-est} then 
gives \eqref{RHS-est'} for $i=1$ and $|m|=k=|s|=1, a=0$  and $|m|=2, k=1, |s|+|a|=2$.


\DETAILS{we find
\begin{align}\label{vrho-kap-est}
& \|[d  g (\rho_{\g}) \rho_{J^s\g}, J^a\k]\|_{W^0}\ls t^{-d(1+\frac1p -\frac1q)}  \|J^{m}\k \|_{W^0}^2 \|\k \|_{W^0},
\end{align}
provided $s+a\le m$ and $p\ge d/m$. This gives \eqref{RHS-est'} for $i=1$ and $k=m=1$.}

Now, we prove \eqref{RHS-est'} for $i=1$ and $|k|=2$,  which implies $2|s|+|a|=|m|$.
We use the assumption \eqref{g-cond2} instead of \eqref{g-cond1} to obtain, for $a=0$, as in \eqref{vrho-kap-est'},
\begin{align}\label{d2g1-est'} 
\begin{split}
  \|[d^2 g_1 (\rho_{\g}) (\rho_{J^s\g})^2, \k']\|_{W^0}
  & \ls t^{-b}\|d^2 g_1 (\rho_{\g}) (\rho_{J^s\g})^2\|_p  \|J^{b}\k' \|_{W^0}
\\
&\ls t^{-  b}\| \rho_{J^s\g}\|_{q'}^2  \|J^{b}\k \|_{W^0}
\\
&\ls  t^{-b-2\nu}  \|J^{b'+s'}\k \|_{W^0}^2 \|J^{b''+s''}\k \|_{W^0}^2 \|J^{b}\k' \|_{W^0},
\end{split}
\end{align}
where $s'+s''=s,$ $\nu:=d(1-\frac{1}{q'})$, $b= d/p$, 
and $b'$ and $b''$ are any non-negative numbers satisfying $b'+b''=d(1-\frac{1}{q'})$.
Since $d(1-\frac{1}{q'})\le |m|$ and $|s|+|a|\le |m|$, we can choose them so that $b'+s', b''+s''\le |m|$, giving
\begin{align}\label{d2g1-est1} 
  \|[d^2 g_1 (\rho_{\g}) (\rho_{J^s\g})^2, \k']\|_{W^0}
  &\ls  t^{-d(2-\frac{2}{q'}+ \frac{1}{p})}  \|\k \|_{W^m}^4 \|\k' \|_{W^m}.
\end{align}
This completes the proof of \eqref{RHS-est'} for $i=1$ and $k=2$ and $a=0$, which suffices for $|m|=2$. 

\medskip
{\it Estimates for $g_2$}.
As above, we rely on the inequality \eqref{L2Ls-est}, but now need a different argument for the estimates
in view of the possible singular nature of the derivatives of the exchange-correlation term $\rho^\beta$.

We prove \eqref{RHS-est'} for $i=2$ and $|m|=k=1$ which implies $|s|=1, a=0$. 
To this end, we need to estimate $\|[d g_2 (\rho_{\g}) (\rho_{J\g}), \k] \|_{W^0}$.
We calculate explicitly
\begin{align} \notag& \| d g_2 (\rho_{\g}) (\rho_{J\g}) \, \k \|_{W^0}^2
= \beta^2 \int_{\R^3\times\R^3}\left| \rho^{\beta-1}_{\g}(x)  \, \rho_{J\g} (x) \, \k(x,y) \right|^2  dxdy
 \\
\notag & = \beta^2 \int_{\R^3} 
 	\rho^{2\beta-2}_{\g}(x)\, 
 	\rho_{J\g}^2 (x) \, ( \int_{\R^3} |\k(x,y)|^2  dy) dx
\\
\label{dg-est1} & = \beta^2 \int_{\R^3} 
 	\rho^{2\beta-1}_{\g}(x)\, 
 	\rho^2_{J\g} (x) \,  dx.
\end{align}

Now, using the relations $\g=\k^\ast \k$ and \eqref{rhoJ-est1} 
in \eqref{dg-est1}, $\rho_{\k^\ast \k} = \| \k \|_{L^2_y}^2$, 
and $ \| \k \|_{L^\infty_x L^2_y}\ls  \| \k \|_{L^2_rL^\infty_c}$, we find
\begin{align}\label{dg-est2}
\| d g_2 (\rho_{\g}) (\rho_{J\g}) \, \k \|_{W^0}^2
  & \le 2 \beta^2 \int_{\R^3} 
 \rho^{2\beta}_{\k^\ast \k}(x)\, 
 \rho_{J\k^\ast J\k} (x) \,  dx \notag \\
&\le 2 \beta^2  \| \k \|_{L^2_rL^\infty_c}^{4\beta} \| J\k \|_{L^2_xL^2_y}^{2}	.
\end{align}
The second factor on the right-hand side of \eqref{dg-est2} is equal $\|J\k \|_{W^0}^2$. 
For the first factor, we use \eqref{L2Ls-est} with $s=\infty$ and $\al=1$
\DETAILS{ \eqref{rho-est}, with $\ka'=\ka$ and $\al=1$, i.e.
\begin{align}\label{rho-est'}
 \| \k \|_{L^2_rL^\infty_c}
  &\ls t^{-b/2} \|\k \|_{W^b},\  b=d/2,
\end{align}\\
and we use the Sobolev-type inequality  \eqref{Sob-type-ineq} for the second one}
to find, for $ b=[d/2]+1$,
\begin{align} \label{dg-est3} 
\| d g_2 (\rho_{\g}) (\rho_{J\g}) \, \k \|_{W^0} & \lesssim 
  t^{-\beta d} \|\k \|_{W^b}^{2\beta} \| J \k \|_{W^0},
\end{align}\\
which  yields \eqref{RHS-est'} for $|m|=1$.

We now consider \eqref{RHS-est'} with  $|m|=2$ 
and $k=1, a=0$ and $|s|=2$.
\DETAILS{ in \eqref{RHS-est}, we need to estimate, among others, the three terms

\begin{align} \label{EX2.1}
&  \| [d g (\rho_{\g}) (\rho_{J^m\g}), \, \k]\|_{W^0}, \quad
	\| [d g (\rho_{\g}) (\rho_{J\g}), \, J^{m-1} \k]\|_{W^0}, \quad
	\| [d^m g (\rho_{\g}) (\rho_{J\g})^m, \, \k]\|_{W^0}.
\end{align}}
We compute as in \eqref{dg-est1} 
\begin{align} \notag \| d g_2 (\rho_{\g}) \rho_{J^s\g} \, \k \|_{W^0}^2
& = \beta^2 \int_{\R^3} \rho^{2\beta-2}_{\g}(x)\, 
 	\rho_{J^s\g}^2 (x) \, ( \int_{\R^3} |\k(x,y)|^2  dy) dx\\
\label{dgJ2-est1} & = \beta^2 \int_{\R^3} \rho^{2\beta-1}_{\g}(x)\, 
 	\rho_{J^s\g}^2 (x) \,  dx.
\end{align}
Using the relation  $J^2 (\k^\ast \k)= (J^2\k^\ast) \k +\k^\ast (J^2\k)+ 2  (J\k^\ast)  J\k)$ 
in \eqref{dgJ2-est1}, and $\beta \geq 1/2$, we find
\begin{align}
\| d g_2 (\rho_{\g}) \rho_{J^s\g} \, \k \|_{W^0}
& \lesssim \| J^s\k \|_{L^2_xL^2_y}  \| \k \|_{L^2_rL^\infty_c}^{2\beta}	
\label{dg2-est2} \\
\label{dg2-est1}
& + \| \k \|_{L^2_rL^\infty_c}^{2\beta-1} \| J\k \|_{L^4_xL^2_y}^2.
\end{align}
The term on the right-hand side of \eqref{dg2-est2} is of the same form obtained in \eqref{dg-est2},
while the term \eqref{dg2-est1} can be estimated using \eqref{xyrc} followed by the usual \eqref{L2Ls-est}: 
\begin{align*}
\| d g_2 (\rho_{\g}) \rho_{J^s\g} \, \k \|_{W^0}
  & \lesssim \| \k \|_{L^2_rL^\infty_c}^{2\beta} \| J^s\k \|_{W^0} +  \| \k \|_{L^2_rL^\infty_c}^{2\beta-1} \| J\k \|_{L^2_rL^4_c}^2
\\
& \lesssim t^{-d\beta} {\| \k \|}_{W^2}^{2\beta} \| J^s\k \|_{W^0}
  + t^{-d/2} \| \k \|_{L^2_rL^\infty_c}^{2\beta-1} {\| J^2 \k \|}_{W^0}^{d/2} {\| J \k \|}_{W^0}^{(4-d)/2}
\\
& \lesssim t^{-d\beta} {\| \k \|}_{W^2}^{2\beta+1}
\end{align*}
having used $\beta\geq 1/2$ in the last inequality.

Now, we consider \eqref{RHS-est'} for $i=2$ with $k=1$ and $|s|=1=|a|$. 
We compute as in \eqref{dg-est1}
\begin{align} \notag\| d g_2 (\rho_{\g}) (\rho_{J \g}) \, J \k \|_{W^0}^2
 & = \beta^2 \int_{\R^3} 
 	\rho^{2\beta-2}_{\g}(x)\, 
 	\rho_{J \g}^2 (x) \, ( \int_{\R^3} |J \k(x,y)|^2  dy) dx\\
\label{dgJJ-est1} & = \beta^2 \int_{\R^3} 
 	\rho^{2\beta-2}_{\g}(x)\, \rho^2_{J \g} (x) \rho_{J\k^\ast J \k} (x)\,  dx.
\end{align}
Using the inequality 
\eqref{rhoJ-est1}, we find
\begin{align*} 
\| d g_2 (\rho_{\g}) (\rho_{J \g}) \, J \k \|_{W^0}^2 & \le 2 \beta^2 \int_{\R^3} 
\rho^{2\beta-1}_{\g}(x)\,  \rho_{J \k^\ast J \k} (x) \rho_{J\k^\ast J \k} (x)\, dx
\\
&\lesssim \beta^2 \| \k \|_{L^2_rL^\infty_c}^{4\beta-2} \| J \k \|_{L^{4}_xL^2_y}^{4}.
\end{align*}
The right-hand side is a product of terms we treated above and we see that
\begin{align} 
\label{dg2-fin}
\| d g_2 (\rho_{\g}) (\rho_{J \g}) \, J \k \|_{W^0} \lesssim t^{-d\beta} \| \k \|_{W^2}^{2\beta+1}.
\end{align}

Finally, we consider \eqref{RHS-est'} for $i=2$ with $k>1, a=0$ and $|s|=1$.
We compute as in \eqref{dg-est1} 
\begin{align} \notag 
\| d^k g_2 (\rho_{\g}) (\rho_{J\g})^k \, \k \|_{W^0}^2
& = (\beta (1-\beta))^2 \int_{\R^3} \rho^{2\beta-2k}_{\g}(x)\, 
  \rho_{J\g}^{2k} (x) \, ( \int_{\R^3} |\k(x,y)|^2  dy) dx
\\
\label{d2g-est1} & = (\beta (1-\beta))^2 \int_{\R^3} \rho^{2\beta-2k +1}_{\g}(x)\, 
  \rho^{2k}_{J\g} (x) \,  dx.
\end{align}
Using \eqref{rhoJ-est1} in \eqref{d2g-est1}, we find, for $\beta\geq(k-1)/2$,
\begin{align*}
\| d^k g_2 (\rho_{\g}) (\rho_{J\g})^k \, \k \|_{W^0}^2
 & \lesssim 
 \int_{\R^3} \rho^{2\beta-k + 1}_{\g}(x)\, \rho_{J\k^\ast J\k}^{k} (x) \, dx 
\\
& \lesssim \| \k \|_{L^2_rL^\infty_c}^{4\beta-2k+2} \| J\k \|_{L^4_xL^2_y}^{2k}.
\end{align*}
The square root of this last quantity is again a product of terms like those treated above and, 
using \eqref{xyrc} and \eqref{L2Ls-est}, 
can be bound by the right-hand side of \eqref{dg2-fin}.
This concludes the proof of \eqref{RHS-est'}-\eqref{RHS-estp'} and the energy estimate \eqref{energy-est}.
\end{proof}

\begin{remark}[Higher dimensions]\label{rem:d>3'}
The calculation done for general $k>1$ in \eqref{d2g-est1} shows that one can close this type 
of estimates even for $k>2$ provided $\beta \geq (k-1)/2$. 
Since we need $k = [d/2]+1$ derivatives to deduce the necessary 
sharp $L^\infty_cL^2_r$ decay (through \eqref{L2Ls-est}),
this means that in dimension $d>3$ it is possible to treat the case of $\mathrm{xc}(\rho) = \rho^\beta$
for $\beta \geq (1/2)[d/2]$.
Of course, when applying $k$ derivatives with $k>2$ there are several other terms to consider besides \eqref{d2g-est1};
however, these other terms can all be treated with similar arguments to those in 
the proof of Lemma \ref{lem:energy-est} above,
using \eqref{xyrc} and proper applications of \eqref{L2Ls-est}. 
\end{remark}

\bigskip
\section{Local existence, GWP and scattering for \eqref{KS-k}}\label{sec:loc-exist}
In this section we will use the non-abelian analogues of Sobolev 
spaces based on the space of Hilbert-Schmidt operators introduced in \eqref{LOC1'},
which we recall here for convenience:
\begin{align}\label{LOC1} 
V^s &:= \big\{ \k \in I^2\, : \, \sum_{|\alpha| \leq s} \| D^\alpha \k  \|_{I^2} < \infty \big\},
\end{align}
with $D_\ell\k := [\partial_{x_\ell} ,\k], D =(D_1,\dots,D_d) = [\n,\,\cdot\,]$, 
for any positive integer $s$.
Note that $V^0  =  I^2 = W^0$, see \eqref{defIr} and \eqref{LOC2}.

\begin{thm}[Local existence]\label{pr:loc1} 
Assume \eqref{g1g2}-\eqref{g-cond1'} and consider  equation \eqref{KS-k} with initial data $\k(0)=\k_0$. Then we have the following: 

\begin{itemize}

\smallskip
\item[(i)] ({\it Local existence}) If $\k_0\in V^{[d/2]+1}$, then there exists $T_0=T_0(\|\k_0\|_{V^{[d/2]+1}})>0$ and a unique solution 
$\k \in C([-T_0,T_0], V^{[d/2]+1})$ of \eqref{KS-k} with $\k(0)=\k_0$.

\smallskip
\item[(ii)] ({\it Energy Estimate}) 
If $\k_0 \in V^k$, $k \geq [d/2]+1$, then the solution $\k \in C([-T_0,T_0], V^{[d/2]+1})$ of  \eqref{KS-k} 
from part (i) satisfies the following energy estimate:
\begin{align}\label{LOCEE}
\frac{d}{dt} {\|\k(t)\|}_{V^k} \leq \lambda |t|^{-p} \cdot P({\| \k(t) \|}_{W^{[d/2]+1}}) 
	\cdot {\|\k(t)\|}_{V^k} 
\end{align}
for some $p>1$ (depending on $g$), where $\lambda = |\lambda_1|+|\lambda_2|$, 
and $P$ is a polynomial with positive coefficients which depend on $g$,$d$ and $k$.

\smallskip
\item[(iii)] ({\it Continuity of the weighted norm}) 
For $\k_0 \in V^k \cap W^b$, with $[d/2]+1\leq b\leq k$, 
the map $t \rightarrow \k(t)$ is continuous from $[-T_0,T_0]$ to $V^k \cap W^b$.

\end{itemize}

\end{thm}

\begin{proof}
{\it (i)}
Denote by $\alpha_t$ the linear flow associated with the operator $\k \mapsto i[-\Delta, \k]$.
Note that $\alpha_t$ is unitary on $V^k$.
We obtain the solution $\kappa$ as a fixed point of the map 
\begin{align}\label{loc0}
\Phi(\k(t)) =  \alpha_t(\k_0) + \int_0^t \alpha_{t-s} (\big[g(\rho(\g(s))),\k(s) \big]) ds 
\end{align}
in the space
\begin{align*}
\big\{ \k \in C([-T_0,T_0], V^{[d/2]+1}), \, \sup_{[-T_0,T_0]} {\| \k(t) \|}_{V^{[d/2]+1}} 
	\leq 2 {\|\k_0\|}_{V^{[d/2]+1}} \big\},
\end{align*}
for a sufficiently small $T_0$.
For this it suffices to prove, for all $k \leq [d/2]+1$, the estimates
\begin{align}\label{loc10}
& {\| \Phi(\k(t)) \|}_{V^{[d/2]+1}} \leq {\| \k_0 \|}_{V^{[d/2]+1}}
	+ \int_0^t  P\big({\|\k(s)\|}_{V^{[d/2]+1}}\big) ds,\\
\label{loc10'} & {\| \Phi(\k_1(t)) - \Phi(\k_2(t)) \|}_{V^{[d/2]+1}} 
\leq \int_0^t  Q\big({\|\k\|}_{V^{[d/2]+1}}\big) {\| \k_1(s) - \k_2(s) \|}_{V^{[d/2]+1}} ds,
\end{align}
for some polynomials  $P$  and $Q$ with positive coefficients.

To prove \eqref{loc10} we first notice that $[D,\alpha_t]=0$ and thus \eqref{loc10} can be reduced to proving
that for all $k \leq [d/2]+1$
\begin{align}\label{loc11}
{\| D^k \big[ g(\rho(\g)),\k \big] \|}_{V^0} \lesssim P\big({\|\k\|}_{V^{[d/2]+1}}\big).
\end{align}
Estimate \eqref{loc11} then follows similarly to the proof of Proposition \ref{lem:energy-est},
(which deals with $J$ and the space $W^k$ instead of $D$ and the space $V^k$).
First we commute $D$ with 
$D_\g\k := \p_t \k - i [h_{\g}, \k]$
(in the same way that we commuted $J$ in Proposition \ref{prop:energy-relat}) to obtain 
\begin{equation}
 D_\g D\k = D D_\g \k + i [d  g (\rho_{\g}) \rho_{D\g},\k]
\end{equation}
We then use 
$\|f \k \|_{V^0} \ls \|f \|_{L^p} \| \k \|_{L^2_rL^s_c}$, $1/p+1/s=1/2$,
see \eqref{fkap-est0}, and \eqref{rho-est'}, 
followed by the Gagliardo-Nirenberg-Sobolev type inequality
\begin{align}\label{loc13}
\|\k \|_{L^2_rL^s_c} &\ls \|\k \|_{V^b}^\al \| \k \|_{V^0}^{1-\al},
\end{align}
for $\al b=d(\frac12-\frac1s)$, $s\ge 2$, see \eqref{GNI-non-ab}, to find \eqref{loc10}.
The proof of the estimate for the differences is similar so we skip the details.

\medskip
{\it (ii)}
In Lemma \ref{lem:energy-est} we proved a  more precise version of \eqref{LOCEE}
with the weighted $W^k$-norm replacing the $V^k$-norm.
Therefore, we leave to the reader the details of the proof of the 
more standard energy inequality \eqref{LOCEE} which follows from similar arguments.

\medskip
{\it (iii)}
This property follows from similar (short-time) energy estimates. 
Continuity of the map $t \in [0,T] \mapsto \k(t) \in V^k$ 
follows essentially from \eqref{loc11} which also shows
\begin{align*}
\frac{d}{dt} \| \k(t) \|_{V^k} \lesssim P\big({\|\k\|}_{V^{[d/2]+1}}\big).
\end{align*}
Continuity in the weighted space, $W^b$, follows from the analogous weighted energy estimate
\begin{align}\label{loc20}
\frac{d}{dt} \| \k(t) \|_{W^k} \lesssim P\big({\|\k\|}_{V^{[d/2]+1}}\big) \| \k(t) \|_{W^k},
\end{align}
which can be obtained by the exact same arguments used 
in the proof of \eqref{lem:energy-est}, making use of \eqref{loc13} instead of \eqref{L2Ls-est}.
\end{proof}

\begin{proof}[Proof of Proposition \ref{prop:LD->ST}] 
In view of item (i) of Theorem \ref{pr:loc1}, in order to continue a local-in-time solution of \eqref{KS-k}
to a global one, it suffices to obtain a uniform in time a priori bound for the $V^k$-norm with $k\geq [d/2]+1$.
This follows by an application of Gronwall's inequality to \eqref{LOCEE} 
with the uniform bound $\|\k(t)\|_{W^{[d/2]+1}} \lesssim \|\k_0\|_{W^{[d/2]+1}}$ 
given by \eqref{Jb-est} in Proposition \ref{prop:AprioriBnds}.

The scattering property for equation \eqref{KS-k} in the space $V^0$ 
(also in $V^k\cap W^b$, $[d/2]+1\leq b \leq k$) 
is proven by standard arguments as follows. Let $\al_t(\ka)$ be the linear evolution 
$\al_t(\k):=e^{i \Delta t} \k  e^{-i \Delta t}$. 
Define $\tilde\k(t):=\al_{-t}(\ka(t))$ and use \eqref{KS-k} to compute
\[\p_t \tilde\k(t)=\al_{-t}(i [g(\rho_{\k^*\k}),\ka(t)]).\]
Writing $\tilde\k(t)$ as the integral of its derivative, 
using the above relation, taking the $(I^2=V^0)$-norm of the resulting identity and using the unitarity of $\al_{-r}$ gives
\begin{align}
\nonumber
{\| \tilde\k(t) - \tilde\k(s)\|}_{V^0} & \ls \int_s^t{\| \al_{-\tau}(i [g(\rho_{\k^*\k}),\ka(\tau)])\|}_{V^0} \, d\tau
\\
\label{ka-t-est}
& \ls \int_s^t{\|  [g(\rho_{\k^*\k}),\ka(\tau)]\|}_{V^0} \, d\tau.
\end{align}
Then we apply estimate \eqref{fkap-est0} with $p=\infty, s=2$,
use the conditions \eqref{g-cond0} and \eqref{g-cond0'} on $g_1$ and $g_2$,
the estimate \eqref{rho-est} (with $b=b'=[d/2]+1$, so that in particular $\alpha,\alpha' < 1$), 
to obtain
\begin{align*}
\nonumber 
{\|  [g(\rho_{\k^*\k}),\ka(\tau)]\|}_{V^0} 
\nonumber 
& \ls {\| g(\rho_{\k^*\k})\|}_\infty  {\| \k(\tau) \|}_{V^0}
\\
\nonumber
& \ls \big( {\|\rho_{\k^*\k}\|}_{q}^a + {\|\rho_{\k^*\k}\|}_\infty^{\beta}  \big) {\| \k(\tau) \|}_{V^0}
\\
& \ls 
  |\tau|^{-d(1-1/q)a} + |\tau|^{-d\beta} 
\end{align*}
where the parameters $a$ and $\beta$ above are those appearing in \eqref{g-cond0} and \eqref{g-condbeta},
and the implicit constant in the last inequality depends on $\| \k(\tau) \|_{W^{[d/2]+1}}$.
Since $a$ and $\beta$ satisfy $d(1-1/q)a$ and $d\beta > 1$, the integrand in \eqref{ka-t-est} is integrable in time. 
Hence $\tilde\k(t)$ has the Cauchy property and therefore converges to some $\k_\infty \in V^0$ as $t\ra \infty$. 
This implies
\begin{align}\label{kapscat'}
& {\| \k(t) - e^{i \Delta t} \k_\infty e^{-i\Delta t}\|}_{V^0} \ra 0.
\end{align}
\end{proof}

\end{document}